\documentclass[dvips]{amsart}

\usepackage[foot]{amsaddr}
\title[Pivotal structures of the Drinfeld center]
{Pivotal structures of the Drinfeld center \\
  of a finite tensor category}
\author[K.~Shimizu]{Kenichi Shimizu}
\email{kshimizu@shibaura-it.ac.jp}
\address{Department of Mathematical Sciences \\
  Shibaura Institute of Technology \\
  307 Fukasaku, Minuma-ku, Saitama-shi, Saitama 337-8570, Japan.}
\thanks{The author is supported by JSPS KAKENHI Grant Number 16K17568}
\date{}

\usepackage{amsmath,amssymb,amscd}
\usepackage{bbm}
\usepackage{stmaryrd}
\usepackage{url}
\usepackage{mathrsfs}
\usepackage{upgreek}
\usepackage{eucal}

\usepackage{amsthm}
\numberwithin{equation}{section}

\newtheorem{counter}{}[section]

\theoremstyle{definition}

\newtheorem*{notation*}         {Notation}

\theoremstyle{plain}
\newtheorem{lemma}              [counter]{Lemma}

\newtheorem{theorem}            [counter]{Theorem}

\newtheorem*{theorem*}          {Theorem}
\theoremstyle{remark}
\newtheorem{remark}             [counter]{Remark}

\newtheorem{claim}              [counter]{Claim}

\DeclareMathOperator{\Hom}{Hom}

\DeclareMathOperator{\Nat}{Nat}
\DeclareMathOperator{\Aut}{Aut}
\DeclareMathOperator{\cAut}{\underline{Aut}}
\DeclareMathOperator{\Alg}{Alg}
\newcommand{\Ens}{\mathord{\mathbf{Set}}}
\newcommand{\unitobj}{\mathord{\mathbbm{1}}}
\DeclareMathOperator{\id}{id}
\DeclareMathOperator{\eval}{ev}
\DeclareMathOperator{\coev}{coev}
\DeclareMathOperator{\op}{op}
\DeclareMathOperator{\rev}{rev}
\DeclareMathOperator{\Inv}{Inv}
\newcommand{\canact}{\mathord{\mathsf{a}}}
\newcommand{\intch}{\mathsf{ch}}

\usepackage[all,knot]{xy}
\newcommand{\mycap}[1]{\POS{:
    (0,0); (0,#1) **@{} ?(.5)="_P1",
    (1,0); (1,#1) **@{} ?(.5)="_P2",
    (0,0); (1,0) **\crv{"_P1" & (0.5,#1) & "_P2"}}}
\newcommand{\pushrect}[2]{\POS{ c="_P1",
    "_P1"+(#1,-#2) @+, "_P1"+( 0,-#2) @+,
    "_P1"+(#1,  0) @+, "_P1"+( 0,  0) @+ }}
\newcommand{\makecoord}[2]{\POS{#1; p+(0, 1) **\dir{} ?!{#2; p+(1, 0)}}}
\newcommand{\mybendWithOpt}[2]{\POS{
    **@{} ?<="_P1" ?>="_P2",
    {"_P1"; p-(0,1) **@{} ?!{"_P2"; p-(1,0)}}="_P3",
    {"_P2"; p-(0,1) **@{} ?!{"_P1"; p-(1,0)}}="_P4",
    "_P1"; "_P3" **@{} ?(#1)="_P5",
    "_P2"; "_P4" **@{} ?(#1)="_P6",
    "_P1"; "_P2" **\crv{#2 "_P5" & "_P6"}}}
\newcommand{\mybend}[1]{\POS{\mybendWithOpt{#1}{}}}
\begin{document}

\maketitle

\begin{abstract}
  We classify the pivotal structures of the Drinfeld center $\mathcal{Z}(\mathcal{C})$ of a finite tensor category $\mathcal{C}$. As a consequence, every pivotal structure of $\mathcal{Z}(\mathcal{C})$ can be obtained from a pair $(\beta, j)$ consisting of an invertible object $\beta$ of $\mathcal{C}$ and an isomorphism $j: \beta \otimes (-) \otimes \beta^* \to (-)^{**}$ of monoidal functors.
\end{abstract}

\section{Introduction}

Throughout this paper, we work over an algebraically closed field $k$ of arbitrary characteristic. By a {\em tensor category} over $k$, we mean a $k$-linear abelian rigid monoidal category satisfying natural conditions \cite{MR3242743}. The rigidity means that every object $X$ in a tensor category has a well-behaved dual object $X^*$. Although the dual object generalizes the contragradient representation in the group representation theory, the `double dual' object $X^{**}$ is no longer isomorphic to $X$ in general. The assignment $X \mapsto X^{**}$ extends to a tensor autoequivalence on $\mathcal{C}$. A {\em pivotal structure} of a tensor category is an isomorphism $\id_{\mathcal{C}} \to (-)^{**}$ of tensor functors. Such an isomorphism does not exist in general, however, we often require a pivotal structure in some applications of tensor categories to, for example, representation theory, low-dimensional topology, and conformal field theory. Thus it is interesting and important to know when a tensor category admits a pivotal structure.

A {\em finite tensor category} \cite{MR2119143} is a tensor category satisfying a certain finiteness condition. In this paper, we classify the pivotal structures of the Drinfeld center of a finite tensor category. To state our main result, we introduce some notations: Let $\mathcal{C}$ be a finite tensor category, and let $\mathcal{Z}(\mathcal{C})$ denote the Drinfeld center of $\mathcal{C}$. The definition of $\mathcal{Z}(\mathcal{C})$ will be recalled in Subsection~\ref{subsec:drinfeld-center}, but here we note that an object of $\mathcal{Z}(\mathcal{C})$ is a pair $\mathbf{V} = (V, \sigma)$ consisting of an object $V \in \mathcal{C}$ and a natural isomorphism $\sigma(X): V \otimes X \to X \otimes V$ ($X \in \mathcal{C}$).

We now denote by $\mathscr{J}$ the class of all pairs $(\beta, j)$ consisting of an invertible object $\beta \in \mathcal{C}$ and a monoidal natural transformation $j: I^{\beta} \to (-)^{**}$, where
\begin{equation*}
  I^{\beta}(X) = \beta \otimes X \otimes \beta^* \quad (X \in \mathcal{C})
\end{equation*}
Given elements $\boldsymbol{\beta} = (\beta, j)$ and $\boldsymbol{\beta}' = (\beta', j')$ of $\mathscr{J}$, we write $\boldsymbol{\beta} \sim \boldsymbol{\beta}'$ if there exists an isomorphism $\beta \to \beta'$ in $\mathcal{C}$ compatible with the monoidal natural transformations $j$ and $j'$. Our main result is that the pivotal structures of $\mathcal{Z}(\mathcal{C})$ is in bijection with the set $\mathscr{J} / \mathord{\sim}$. More precisely, we prove the following theorem:

\begin{theorem}
  \label{thm:classify-piv-str}
  For an element $\boldsymbol{\beta} = (\beta, j) \in \mathscr{J}$, we define the morphism
  \begin{equation*}
    \Phi(\boldsymbol{\beta})_{\mathbf{V}}: \mathbf{V} \to \mathbf{V}^{**}
    \quad (\mathbf{V} = (V, \sigma) \in \mathcal{Z}(\mathcal{C}))
  \end{equation*}
  by the composition
  \begin{equation*}
    V
    \xrightarrow{\quad \id_V \otimes \coev_{\beta} \quad}
    V \otimes \beta \otimes \beta^*
    \xrightarrow{\quad \sigma(\beta) \otimes \id_{\beta^*} \quad}
    \beta \otimes V \otimes \beta^*
    \xrightarrow{\quad j_V \quad} V^{**},
  \end{equation*}
  where $\coev_{\beta}$ is the coevaluation \textup{(}see Subsection~\ref{subsec:rigidity} for our convention\textup{)}. The assignment $\boldsymbol{\beta} \mapsto \Phi(\boldsymbol{\beta})$ induces a bijection between the set $\mathscr{J} / \mathord{\sim}$ and the set of the pivotal structures of $\mathcal{Z}(\mathcal{C})$.
\end{theorem}

The map $\boldsymbol{\beta} \mapsto \Phi(\boldsymbol{\beta})$ can be defined for arbitrary rigid monoidal categories, but it doesn't seem to be a bijection in general. Techniques of Hopf (co)monads, initiated and organized in \cite{MR2355605,MR2869176,MR2793022}, are essentially used to construct the inverse of the map $\boldsymbol{\beta} \mapsto \Phi(\boldsymbol{\beta})$.

This paper is organized as follows: In Section~\ref{sec:prelim}, we collect some basic results on monoidal categories from \cite{MR1712872,MR3242743} and fix some notations used throughout in this paper. In Section~\ref{sec:central-Hopf-comonad}, we recall the notion of Hopf comonad and the fundamental theorem for Hopf modules over a Hopf comonad. We then use it to prove some technical lemmas on the `adjoint algebra' in a finite tensor category $\mathcal{C}$, which naturally arise from a right adjoint of the forgetful functor $\mathcal{Z}(\mathcal{C}) \to \mathcal{C}$.

In Section~\ref{sec:main-result}, we prove Theorem~\ref{thm:classify-piv-str} in a more general form. For tensor autoequivalences $F$ and $G$ of $\mathcal{C}$, we denote by $\mathscr{J}(F, G)$ the class of all pairs $(\beta, j)$ consisting of an invertible object $\beta \in \mathcal{C}$ and a monoidal natural transformation $j: I^{\beta} F \to G$. There is an equivalence relation $\sim$ on $\mathscr{J}(F, G)$ defined in a similar way as above. We establish a bijection between the set $\mathscr{J}(F, G) / \mathord{\sim}$ and the set of monoidal natural transformations from $\widetilde{F}$ to $\widetilde{G}$, where $\widetilde{F}$ and $\widetilde{G}$ are braided tensor autoequivalences of $\mathcal{Z}(\mathcal{C})$ induced by $F$ and $G$, respectively (Theorem~\ref{thm:main-thm}). Theorem~\ref{thm:classify-piv-str} is actually the special case where $F = \id_{\mathcal{C}}$ and $G = (-)^{**}$. The detail of the proof of Theorem~\ref{thm:main-thm} is too technical to explain here; see the outline of Section~\ref{sec:main-result} given in \S\ref{subsec:outline-main-result}.

In Section~\ref{sec:applications}, we use Theorem~\ref{thm:main-thm} to show that the group $\Aut_{\otimes}(\id_{\mathcal{Z}(\mathcal{C})})$ of automorphisms of the tensor functor $\id_{\mathcal{Z}(\mathcal{C})}$ is isomorphic to the group $\Inv(\mathcal{Z}(\mathcal{C}))$ of the isomorphism classes of invertible objects of $\mathcal{Z}(\mathcal{C})$ (Theorem~\ref{thm:Aut-id-ZC}). We also obtain an exact sequence
\begin{equation*}
  1 \to \Aut_{\otimes}(\id_{\mathcal{C}})
  \to \Inv(\mathcal{Z}(\mathcal{C}))
  \to \Inv(\mathcal{C})
  \to \Aut_{\otimes}(\mathcal{C})
  \to \mathrm{BrPic}(\mathcal{C})
\end{equation*}
of groups involving the Brauer-Picard group (Theorem~\ref{thm:rose-zeli}). This is a generalization of a categorification of the Rosenberg-Zelinsky exact sequence for fusion categories given in \cite{MR3606516}.

\subsection*{Acknowledgment}

The author is supported by JSPS KAKENHI Grant Number 16K17568.

\section{Preliminaries}
\label{sec:prelim}

\subsection{Monoidal categories}
\label{subsec:monoidal-categories}

A {\em monoidal category} \cite[VII.1]{MR1712872} is a category $\mathcal{C}$ endowed with a functor $\otimes: \mathcal{C} \times \mathcal{C} \to \mathcal{C}$ (called the {\em tensor product}), an object $\unitobj \in \mathcal{C}$ (called the {\em unit object}), and natural isomorphisms
\begin{equation*}
  (X \otimes Y) \otimes Z \cong X \otimes (Y \otimes Z)
  \quad \text{and} \quad
  \unitobj \otimes X \cong X \cong X \otimes \unitobj
  \quad (X, Y, Z \in \mathcal{C})
\end{equation*}
satisfying the pentagon and the triangle axioms. If these natural isomorphisms are identities, then $\mathcal{C}$ is said to be {\em strict}.

Let $\mathcal{C}$ and $\mathcal{D}$ be monoidal categories. A {\em monoidal functor} \cite[XI.2]{MR1712872} from $\mathcal{C}$ to $\mathcal{D}$ is a functor $F: \mathcal{C} \to \mathcal{D}$ endowed with a natural transformation
\begin{equation*}
  F_2(X, Y): F(X) \otimes F(Y) \to F(X \otimes Y) \quad (X, Y \in \mathcal{C})
\end{equation*}
and a morphism $F_0: \unitobj \to F(\unitobj)$ in $\mathcal{C}$ satisfying certain conditions. A monoidal functor $F$ is said to be {\em strong} if the structure morphisms $F_2$ and $F_0$ are invertible, and said to be {\em strict} if $F_2$ and $F_0$ are identities.

Let $F$ and $G$ be monoidal functors from $\mathcal{C}$ to $\mathcal{D}$. A {\em monoidal natural transformation} from $F$ to $G$ is a natural transformation $\xi: F \to G$ between underlying functors satisfying $\xi_{\unitobj} \circ F_0 = G_0$ and $\xi_{X \otimes Y} \circ F_2(X, Y) = G_2(X, Y) \circ (\xi_X \otimes \xi_Y)$ for all objects $X, Y \in \mathcal{C}$. We denote by $\Nat_{\otimes}(F, G)$ the set of all monoidal natural transformations from $F$ to $G$ (when $\mathcal{C}$ is essentially small).

A monoidal functor $F: \mathcal{C} \to \mathcal{D}$ is said to be a {\em monoidal equivalence} if  there is a monoidal functor $\overline{F}: \mathcal{D} \to \mathcal{C}$ such that $\overline{F} F \cong \id_{\mathcal{C}}$ and $F \overline{F} \cong \id_{\mathcal{D}}$ as monoidal functors. It is well-known that a strong monoidal functor is a monoidal equivalence if and only if its underlying functor is an equivalence.

Two monoidal categories are said to be monoidally equivalent if there exists a monoidal equivalence between them. The Mac Lane coherence theorem states that every monoidal category is monoidally equivalent to a strict one. For simplicity, we assume that all monoidal categories are strict.

\subsection{Module categories}
\label{subsec:module-categories}

Let $\mathcal{C}$ be a monoidal category. A {\em left $\mathcal{C}$-module category} is a category $\mathcal{M}$ endowed with a functor $\ogreaterthan: \mathcal{C} \times \mathcal{M} \to \mathcal{M}$ (called the {\em action}) and natural isomorphisms $(X \otimes Y) \ogreaterthan M \cong X \ogreaterthan (Y \ogreaterthan M)$ and $\unitobj \ogreaterthan M \cong M$ satisfying the axioms similar to those of monoidal categories. A {\em left $\mathcal{C}$-module functor} is a functor between left $\mathcal{C}$-module categories preserving the actions. The notions of a right module category and functors between them are defined analogously. See \cite[Chapter 7]{MR3242743} for the precise definitions.

Let $\mathcal{M}$ be a left $\mathcal{C}$-module category. If $A$ is an algebra in $\mathcal{C}$ ($=$ a monoid in $\mathcal{C}$ \cite[VII.3]{MR1712872}), then the endofunctor $T := A \ogreaterthan (-)$ is a monad on $\mathcal{M}$. We denote by ${}_A \mathcal{M}$ the category of $T$-modules ($=$ the Eilenberg-Moore category of $T$) and refer to an object of ${}_A \mathcal{M}$ as a {\em left $A$-module in $\mathcal{M}$}. The category $\mathcal{N}_A$ of right $A$-modules in a right $\mathcal{C}$-module category $\mathcal{N}$ is defined in a similar way.

\subsection{Rigidity}
\label{subsec:rigidity}

Let $\mathcal{C}$ be a monoidal category, let $L$ and $R$ be objects of $\mathcal{C}$, and let $\varepsilon: L \otimes R \to \unitobj$ and $\eta: \unitobj \to R \otimes L$ be morphisms in $\mathcal{C}$. We say that $(L, \varepsilon, \eta)$ is a {\em left dual object} of $R$ and $(R, \varepsilon, \eta)$ is a {\em right dual object} of $L$ if the equations
\begin{equation*}
  (\varepsilon \otimes \id_L) \circ (\id_L \otimes \eta) = \id_L
  \quad \text{and} \quad
  (\id_R \otimes \varepsilon) \circ (\eta \otimes \id_R) = \id_R
\end{equation*}
hold. If this is the case, the morphisms $\varepsilon$ and $\eta$ are called the {\em evaluation} and the {\em coevaluation}, respectively. We say that the monoidal category $\mathcal{C}$ is {\em rigid} if every object of $\mathcal{C}$ has a left dual object and a right dual object.

Now we suppose that $\mathcal{C}$ is rigid. We denote by $(X^*, \eval_X, \coev_X)$ the left dual object of $X \in \mathcal{C}$. It is known that the assignment $X \mapsto X^*$ extends to a monoidal equivalence from $\mathcal{C}^{\op}$ to $\mathcal{C}^{\rev}$, where $\mathcal{C}^{\rev}$ is the monoidal category obtained from $\mathcal{C}$ by reversing the order of the tensor product. For simplicity, we assume that $(-)^*$ is a {\em strict} monoidal functor; see \cite[Lemma~5.4]{MR3314297} for a discussion.

We often use the graphical calculus to represent morphisms in a rigid monoidal category; see, {\it e.g.}, \cite{MR1321145}. Our convention is that a morphism goes from the top to the bottom of the diagram. The evaluation and the coevaluation are expressed, respectively, as follows:
\begin{equation*}
  \newcommand{\xylabel}[1]{{\scriptstyle\mathstrut #1}}
  \eval_X =
  \begin{array}{c}
    {\xy /r.75pc/: (0,1); p+(2,0) \mycap{-1}
    ?< *+!D{\scriptstyle \phantom{^*}X^*}
    ?> *+!D{\scriptstyle X} \endxy}
  \end{array}
  \text{\quad and \quad}
  \coev_X =
  \begin{array}{c}
    {\xy /r.75pc/: (0,-1); p+(2,0) \mycap{+1}
    ?< *+!U{\scriptstyle X}
    ?> *+!U{\scriptstyle \phantom{^*}X^*} \endxy}
  \end{array}
\end{equation*}

\subsection{The Drinfeld center}
\label{subsec:drinfeld-center}

Let $\mathcal{C}$ be a rigid monoidal category. The {\em Drinfeld center} of $\mathcal{C}$ is the category $\mathcal{Z}(\mathcal{C})$ defined as follows: An object of this category is a pair $(V, \sigma_V)$ consisting of an object $V \in \mathcal{C}$ and a natural transformation
\begin{equation*}
  \sigma_V(X): V \otimes X \to X \otimes V \quad (X \in \mathcal{C}),
\end{equation*}
called the {\em half-braiding}, such that the equations
\begin{equation}
  \label{eq:half-br}
  \sigma_V(\unitobj) = \id_{V}
  \quad \text{and} \quad
  \sigma_V(X \otimes Y) = (\id_X \otimes \sigma_V(Y)) (\sigma_{V}(X) \otimes \id_Y)
\end{equation}
hold for all $X, Y \in \mathcal{C}$. If $\mathbf{V} = (V, \sigma_V)$ and $\mathbf{W} = (W, \sigma_W)$ are objects of $\mathcal{Z}(\mathcal{C})$, then a {\em morphism} from $\mathbf{V}$ to $\mathbf{W}$ is a morphism $f: V \to W$ in $\mathcal{C}$ satisfying
\begin{equation*}
  \sigma_V(X) \circ (f \otimes \id_X) = (\id_X \otimes f) \circ \sigma_W(X)
\end{equation*}
for all objects $X \in \mathcal{C}$.

Thanks to the rigidity of $\mathcal{C}$, a half-braiding is automatically invertible. Indeed, if $(V, \sigma_V)$ is an object of $\mathcal{Z}(\mathcal{C})$, the inverse of $\sigma_V$ is given by
\begin{equation*}
  \sigma_V(X)^{-1} = (\eval'_X \otimes \id_{V} \otimes \id_X)
  (\id_X \otimes \sigma_V({}^* \! X) \otimes \id_X)
  (\id_X \otimes \id_V \otimes \coev'_X)
\end{equation*}
for $X \in \mathcal{C}$, where $({}^* \! X, \eval_X', \coev'_X)$ is a right dual object of $X$. In particular, since $(X, \eval_X, \coev_X)$ is a right dual object of $X^*$, we have
\begin{equation}
  \label{eq:half-br-inv}
  \sigma_V(X^*)^{-1} = (\eval_X \otimes \id_{V} \otimes \id_X)
  (\id_X \otimes \sigma_V(X) \otimes \id_X)
  (\id_X \otimes \id_V \otimes \coev_X)
\end{equation}
for $X \in \mathcal{C}$. If we represent $\sigma_V(X)$ and its inverse by
\begin{equation*}
  \knotholesize{1em}
  \sigma_V(X) =
  {\xy /r.5pc/:
    (0,1.5) \pushrect{3}{3} \vtwist~{s0}{s1}{s2}{s3},
    s0 *+!D{\scriptstyle V}, s1 *+!D{\scriptstyle X},
    s2 *+!U{\scriptstyle X}, s3 *+!U{\scriptstyle V},
    \endxy}
  \quad \text{and} \quad
  \sigma_V(X)^{-1} =
  {\xy /r.5pc/:
    (0,1.5) \pushrect{3}{3} \vtwistneg~{s0}{s1}{s2}{s3},
    s0 *+!D{\scriptstyle X}, s1 *+!D{\scriptstyle V},
    s2 *+!U{\scriptstyle V}, s3 *+!U{\scriptstyle X},
    \endxy}
  \quad (X \in \mathcal{C}),
\end{equation*}
respectively, then the equations~\eqref{eq:half-br} and \eqref{eq:half-br-inv} are seen as follows:
\begin{equation*}
  \knotholesize{8pt}
  {\xy /r.75pc/:
    (0,1.5) \pushrect{3}{3} \vtwist~{s0}{s1}{s2}{s3},
    s0 *+!D{\scriptstyle {V}}, s1 *+!D{\scriptstyle {X \smash{\otimes} Y}},
    s2 *+!U{\scriptstyle {X \smash{\otimes} Y}}, s3 *+!U{\scriptstyle {V}},
    \endxy}
  \mathop{=}^{\text{\eqref{eq:half-br}}}
  {\xy /r.75pc/:
    (0,1.5) \pushrect{2}{1.5},
    \vtwist~{s0}{s1}{s2}{s3},
    s0 *+!D{\scriptstyle {V}},
    s1 *+!D{\scriptstyle {X}},
    s1+(2,0)="P3" *+!D{\scriptstyle {Y}},
    "P3"; p-(0,1.5) **\dir{-},
    s3 \pushrect{2}{1.5}, \vtwist~{s0}{s1}{s2}{s3},
    s2 *+!U{\scriptstyle {Y}},
    s3 *+!U{\scriptstyle {V}},
    s6; \makecoord{p}{s2} **\dir{-}
    ?> *+!U{\scriptstyle {X}},
    \endxy}
  \qquad
  {\xy /r.75pc/:
    (0,1.5) \pushrect{3}{3} \vtwistneg~{s0}{s1}{s2}{s3},
    s0 *+!D{\scriptstyle X\smash{^*}}, s1 *+!D{\scriptstyle V},
    s2 *+!U{\scriptstyle V}, s3 *+!U{\scriptstyle X\smash{^*}},
    \endxy}
  \mathop{=}^{\text{\eqref{eq:half-br-inv}}}
  {\xy /r1pc/:
    (0,1)="P1"; p-(0,2)="P2",
    "P1"; p-(0,.5) **\dir{-}
    ?> \pushrect{1}{1}, \vtwist~{s0}{s1}{s2}{s3},
    s1; p+(1.5,0) \mycap{.75} ?>; \makecoord{p}{"P2"} **\dir{-}
    ?> *+!U{\scriptstyle \,\,X\smash{^*}},
    s2; p-(1.5,0) \mycap{.75} ?>; \makecoord{p}{"P1"} **\dir{-}
    ?> *+!D{\scriptstyle \,\,X\smash{^*}},
    "P1" *+!D{\scriptstyle V},
    s3; \makecoord{p}{"P2"} **\dir{-}
    ?> *+!U{\scriptstyle V},
    \endxy}
\end{equation*}

Given two objects $\mathbf{V} = (V, \sigma_V)$ and $\mathbf{W} = (W, \sigma_W)$ of $\mathcal{Z}(\mathcal{C})$, their tensor product is defined by $\mathbf{V} \otimes \mathbf{W} = (V \otimes W, \sigma_{V \otimes W})$, where $\sigma_{V \otimes W}$ is given by
\begin{equation*}
  \knotholesize{1em}
  \sigma_{V \otimes W}(X) =
  \begin{array}{c}
    {\xy /r.75pc/:
    (0,1.5)="P1"; p+(2,0) \pushrect{2}{1.5},
    \vtwist~{s0}{s1}{s2}{s3},
    "P1" *+!D{\scriptstyle {V}},
    s0 *+!D{\scriptstyle {W}},
    s1 *+!D{\scriptstyle {X}},
    "P1"; p-(0,1.5) **\dir{-}
    ?> \pushrect{2}{1.5}, \vtwist~{s0}{s1}{s2}{s3},
    s2 *+!U{\scriptstyle {X}},
    s3 *+!U{\scriptstyle {V}},
    s7; \makecoord{p}{s2} **\dir{-}
    ?> *+!U{\scriptstyle {W}},
    \endxy}
  \end{array}
  \quad (X \in \mathcal{C}).
\end{equation*}
The category $\mathcal{Z}(\mathcal{C})$ is in fact a braided rigid monoidal category with respect to this tensor product. See \cite[Section 7.13]{MR3242743} for details. For later use, we recall that a description of a left dual object. Given $\mathbf{V} = (V, \sigma_V) \in \mathcal{Z}(\mathcal{C})$, we define $\mathbf{V}^* \in \mathcal{Z}(\mathcal{C})$ by $\mathbf{V} = (V^*, \sigma_{V^*})$, where the half-braiding $\sigma_{V^*}$ is given by
\begin{equation}
  \label{eq:half-br-left-dual}
  \knotholesize{8pt}
  \sigma_{V^*}(X) = {\xy /r1pc/:
    (0,1)="P1"; p-(0,2)="P2",
    "P1"; p-(0,.5) **\dir{-}
    ?> \pushrect{1}{1}, \vtwistneg~{s0}{s1}{s2}{s3},
    s1; p+(1.5,0) \mycap{.75} ?>; \makecoord{p}{"P2"} **\dir{-}
    ?> *+!U{\scriptstyle \,\,V\smash{^*}},
    s2; p-(1.5,0) \mycap{.75} ?>; \makecoord{p}{"P1"} **\dir{-}
    ?> *+!D{\scriptstyle \,\,V\smash{^*}},
    "P1" *+!D{\scriptstyle X},
    s3; \makecoord{p}{"P2"} **\dir{-}
    ?> *+!U{\scriptstyle X},
    \endxy}
\end{equation}
for $X \in \mathcal{C}$. Then the triple $(\mathbf{V}^*, \eval_V, \coev_V)$ is a left dual object of $\mathbf{V}$.

\begin{remark}
  \label{rem:Dri-cen-double-dual}
  By \eqref{eq:half-br-inv} and \eqref{eq:half-br-left-dual}, we have $\sigma_{V^*}(X^*) = \sigma_V(X)^*$. In particular, the half-braiding $\sigma_{V^{**}}$ of $\mathbf{V}^{**}$ satisfies $\sigma_{V^{**}}(X^{**}) = \sigma_V(X)^{**}$ for all $X \in \mathcal{C}$.
\end{remark}

\subsection{Duality transformation}
\label{subsec:duality-transform}

Let $\mathcal{C}$ and $\mathcal{D}$ be rigid monoidal categories, and let $F: \mathcal{C} \to \mathcal{D}$ be a strong monoidal functor. If $X \in \mathcal{C}$ is an object, then $F(X^*)$ is a left dual object of $F(X)$ with the evaluation and the coevaluation given by
\begin{equation}
  \label{eq:def-duality-trans-0}
  e_X = F_0^{-1} \circ F(\eval_X) \circ F_2(X^*, X)
  \quad \text{and} \quad
  c_X = F_2(X, X^*)^{-1} \circ F(\coev_X) \circ F_0,
\end{equation}
respectively. The {\em duality transformation} \cite[Section 1]{MR2381536} of $F$ is the natural isomorphism $\gamma^F_X: F(X^*) \to F(X)^*$ determined by either of the equations
\begin{gather}
  \label{eq:def-duality-trans}
  e_X = \eval_{F(X)} \circ (\gamma^F_X \otimes \id_{F(X)})
  \quad \text{or} \quad
  \coev_{F(X)} = (\id_{F(X)} \otimes \gamma^F_X) \circ c_X
\end{gather}
for all $X \in \mathcal{C}$. We note that the assignments $X \mapsto F(X^*)$ and $X \mapsto F(X)^*$ are strong monoidal functors from $\mathcal{C}^{\op}$ to $\mathcal{D}^{\rev}$. The duality transformation $\gamma^F$ is in fact an isomorphism of monoidal functors \cite{MR2381536}.

Let $G: \mathcal{C} \to \mathcal{D}$ be another strong monoidal functor. The following well-known lemma is obtained by interpreting \cite[Proposition 7.1]{MR1250465} in our notation.

\begin{lemma}
  \label{lem:mon-nat-inv}
  Every monoidal natural transformation $j: F \to G$ is invertible with the inverse determined by the following equation:
  \begin{equation*}
    (j_X^{-1})^* \circ \gamma^F_X = \gamma^G_X \circ j_{X^*} \quad (X \in \mathcal{C}).
  \end{equation*}
\end{lemma}

\subsection{Invertible objects}
\label{subsec:inv-obj}

Let $\mathcal{C}$ be a rigid monoidal category, and let $\beta \in \mathcal{C}$ be an object. Then the endofunctor $I^{\beta} := \beta \otimes (-) \otimes \beta^*$ on $\mathcal{C}$ is a monoidal functor with
\begin{equation*}
  I^{g}_0 = \coev_{\beta}
  \quad \text{and} \quad
  I^{\beta}_2(X, Y) = \id_{\beta} \otimes \id_X \otimes \eval_{\beta} \otimes \id_Y \otimes \id_{\beta^*}
  \quad (X, Y \in \mathcal{C}).
\end{equation*}
An object $\beta \in \mathcal{C}$ is said to be {\em invertible} if $\eval_{\beta}$ and $\coev_{\beta}$ are isomorphisms. If $\beta$ is invertible, then $I^{\beta}$ is a strong monoidal functor. For later use, we compute the duality transformation of this strong monoidal functor:

\begin{lemma}
  \label{lem:conj-beta-duality-trans}
  Let $\beta \in \mathcal{C}$ be an invertible object. Then the duality transformation of the strong monoidal functor $I^{\beta}$ is given by
  \begin{equation*}
    I^{\beta}(X^*) = \beta \otimes X^* \otimes \beta^*
    \xrightarrow{\quad c \otimes \id \otimes \id \quad}
    \beta^{**} \otimes X^* \otimes \beta^*
    = I^{\beta}(X)^*
    \quad (X \in \mathcal{C}),
  \end{equation*}
  where $c: \beta \to \beta^{**}$ is the isomorphism determined by either of
  \begin{equation*}
    \eval_{\beta^*} \circ (c \otimes \id_{\beta^*})
    = \coev_{\beta}^{-1}
    \quad \text{or} \quad
    \coev_{\beta^*}
    = (\id_{\beta^*} \otimes c) \circ \eval_{\beta}^{-1}.
  \end{equation*}
\end{lemma}
\begin{proof}
  The morphism $e_X$ defined by~\eqref{eq:def-duality-trans-0} with $F = I^{\beta}$ is given by
  \begin{align*}
    e_X & = \coev_{\beta}^{-1}
    \circ (\id_{\beta} \otimes \eval_X \otimes \id_{\beta^*})
    \circ (\id_{\beta} \otimes \id_{X^*} \otimes \eval_{\beta} \otimes \id_X \otimes \id_{\beta^*}) \\
    & = \eval_{\beta^*} \circ (c \otimes \id_{\beta^*}) \circ (\id_{\beta} \otimes \eval_{\beta \otimes X} \otimes \id_{\beta^*}) \\
    & = \eval_{\beta \otimes X \otimes \beta^*} \circ (c \otimes \id_{X^*} \otimes \id_{\beta^*} \otimes \id_{\beta} \otimes \id_X \otimes \id_{\beta^*}).
  \end{align*}
  Thus, by~\eqref{eq:def-duality-trans}, the duality transformation of $I^{\beta}$ is given as stated.
\end{proof}

\subsection{Finite tensor categories}

Throughout this paper, we work over an algebraically closed field $k$ of arbitrary characteristic. By an algebra over $k$, we always mean an associative and unital algebra over $k$. Given an algebra $A$ over $k$, we denote by $A\mbox{\rm -mod}$ the category of finite-dimensional left $A$-modules. A {\em finite abelian category} over $k$ is a $k$-linear category that is $k$-linearly equivalent to $A\mbox{\rm -mod}$ for some finite-dimensional algebra $A$ over $k$.

A {\em finite tensor category} over $k$ \cite{MR2119143,MR3242743} is a rigid monoidal category $\mathcal{C}$ such that $\mathcal{C}$ is a finite abelian category over $k$, the tensor product of $\mathcal{C}$ is $k$-linear in each variable, and the unit object $\unitobj \in \mathcal{C}$ is a simple object. By a {\em tensor functor}, we mean a $k$-linear exact strong monoidal functor between finite tensor categories. A morphism of tensor functors is just a monoidal natural transformation.

Let $\mathcal{C}$ be a finite tensor category over $k$. We remark some non-trivial facts on invertible objects of $\mathcal{C}$. First, an invertible object of $\mathcal{C}$ is a simple object. Hence the isomorphism classes of invertible objects of $\mathcal{C}$ is finite. Second, an object $V \in \mathcal{C}$ is invertible if and only if there exists an object $W \in \mathcal{C}$ such that $V \otimes W \cong \unitobj$; see, {\it e.g.}, \cite[Chapter 4]{MR3242743}.

\subsection{End}

Let $\mathcal{A}$ and $\mathcal{C}$ be categories, and let $H: \mathcal{A}^{\op} \times \mathcal{A} \to \mathcal{C}$ be a functor. A {\em dinatural transformation} from an object $E \in \mathcal{C}$ to $H$ is a family
\begin{equation*}
  \{ \pi(X) \in \Hom_{\mathcal{C}}(E, H(X,X)) \}_{X \in \mathcal{A}}
\end{equation*}
of morphisms in $\mathcal{C}$ that is {\em dinatural} in the sense that the equation
\begin{equation*}
  H(f, \id_Y) \circ \pi(Y) = H(\id_X, f) \circ \pi(X)
\end{equation*}
holds for all morphisms $f: X \to Y$ in $\mathcal{A}$. An {\em end} of the functor $H$ is a pair $(E, \pi)$ consisting of an object $E \in \mathcal{C}$ and a dinatural transformation $\pi$ from $E$ to $H$ that is `universal' among such pairs. The universal property implies that, if it exists, an end of $H$ is unique up to unique isomorphism. We denote by $\int_{X \in \mathcal{A}} H(X,X)$
the end of $H$; see \cite[IX]{MR1712872} for more detail.

Now let $\mathcal{B}$ be another category, and let $R: \mathcal{A} \to \mathcal{B}$ and $T: \mathcal{B}^{\op} \times \mathcal{A} \to \mathcal{C}$ be functors. Suppose that $R$ admits a left adjoint $L: \mathcal{B} \to \mathcal{A}$ with unit $\eta: \id_{\mathcal{B}} \to R L$ and counit $\varepsilon: L R \to \id_{\mathcal{A}}$. Suppose, moreover, that the end
\begin{equation*}
  E_R := \int_{X \in \mathcal{A}} T(R(X), X)
\end{equation*}
exists. Let $\pi(X): E_R \to T(R(X), X)$ be the universal dinatural transformation of the end $E_R$. We define a morphism $\pi'(Y)$ by
\begin{equation*}
  \pi'(Y) := \left(
    E_R
    \xrightarrow{\quad \pi(L(Y)) \quad}
    T(R L(Y), L(Y))
    \xrightarrow{\quad T(\eta_Y^{}, \id_{L(Y)}^{}) \quad}
    T(Y, L(Y))
  \right)
\end{equation*}
for $Y \in \mathcal{C}$. Then $(E_R, \pi')$ is an end of the functor $(Y_1, Y_2) \mapsto T(Y_1, L(Y_2))$ (the dual of \cite[Lemma 3.9]{MR2869176}). We may write this fact as follows:
\begin{equation}
  \label{eq:end-adj}
  \int_{X \in \mathcal{A}} T(R(X), X)
  \cong \int_{Y \in \mathcal{B}} T(Y, L(Y)).
\end{equation}
The following formula, which is obtained by~\eqref{eq:end-adj}, is useful:

\begin{lemma}
  \label{lem:end-equiv-twist}
  Let $H: \mathcal{A}^{\op} \times \mathcal{A} \to \mathcal{V}$ be a functor. Suppose that an end of $H$ exists. Then, for every autoequivalence $G$ of $\mathcal{A}$, we have
  \begin{equation}
    \label{eq:end-equiv}
    \int_{X \in \mathcal{A}} H(X, X)
    \cong
    \int_{X \in \mathcal{A}} H(G(X), G(X)),
  \end{equation}
  meaning that the end in the right-hand side exists and canonically isomorphic to the end of $H$.
\end{lemma}
\begin{proof}
  Let $\overline{G}$ be a quasi-inverse of $G$. Since $G \dashv \overline{G}$, we have
  \begin{equation*}
    \int_{X \in \mathcal{A}} H(X, X)
    \cong
    \int_{X \in \mathcal{A}} H(G \overline{G}(X), X)
    \cong
    \int_{X \in \mathcal{A}} H(G(X), G(X))
  \end{equation*}
  by equation~\eqref{eq:end-adj} with $T = H(G(-), -)$.
\end{proof}

Let $E$ be the end of $H$, and let $\pi(X): E \to H(X, X)$ be the universal dinatural transformation of the end $E$. Given an autoequivalence $G$ of $\mathcal{A}$, we set $\pi^G(X) = \pi(G(X))$ and $H^G(X, Y) = H(G(X), G(Y))$ for $X, Y \in \mathcal{C}$. By the construction of the canonical isomorphism~\eqref{eq:end-equiv}, we see that $(E, \pi^G)$ is an end of $H^G$. We will use this consequence in the following way: Given an object $E' \in \mathcal{C}$ and a dinatural transformation $\xi: E' \to H^G$, there exists a unique morphism $f: E' \to E$ such that $\xi(X) = \pi^G(X) \circ f$ for all objects $X \in \mathcal{A}$.

\section{The central Hopf comonad}
\label{sec:central-Hopf-comonad}

\subsection{Hopf comonads}

Let $\mathcal{C}$ be a monoidal category. A {\em monoidal comonad} on $\mathcal{C}$ is a comonad $T = (T, \delta, \varepsilon)$ on $\mathcal{C}$ such that the endo\-functor $T: \mathcal{C} \to \mathcal{C}$ is a monoidal functor and the comultiplication $\delta: T \to T T$ and the counit $\varepsilon: T \to \id_{\mathcal{C}}$ of $T$ are monoidal natural transformations.

Let $T$ be a monoidal comonad on $\mathcal{C}$. A {\em $T$-comodule} is a pair $(V, \rho)$ consisting of an object $V$ of $\mathcal{C}$ and a morphism $\rho: V \to T(V)$ such that the equations
\begin{equation}
  \label{eq:def-T-comodule}
  \delta_V \circ \rho = T(\rho) \circ \rho
  \quad \text{and} \quad
  \varepsilon_V \circ \rho = \id_V
\end{equation}
hold. We denote by $\mathcal{C}^T$ the category of $T$-comodules ($=$ the Eilenberg-Moore category of the comonad $T$). The category $\mathcal{C}^T$ is in fact a monoidal category with respect to the tensor product given by
\begin{equation}
  \label{eq:def-T-comod-tensor}
  (M, \rho_M) \otimes (N, \rho_N) = (M \otimes N, T_2(M,N) \circ (\rho_M \otimes \rho_N))
\end{equation}
for $(M, \rho_M), (N, \rho_N) \in \mathcal{C}^T$. The unit object of $\mathcal{C}^T$ is $\unitobj^T := (\unitobj, T_0)$.

A {\em Hopf comonad} on $\mathcal{C}$ ({\it cf}. \cite{MR2793022}) is a monoidal comonad $T$ on $\mathcal{C}$ such that the morphisms $T_2(X, T(Y)) \circ (\id_{T(X)} \otimes \delta_Y)$ and $T_2(T(X), Y) \circ (\delta_X \otimes \id_{T(Y)})$ are invertible for all objects $X, Y \in \mathcal{C}$. Since a Hopf comonad on $\mathcal{C}$ is just a Hopf monad on $\mathcal{C}^{\op}$ in the sense of \cite{MR2793022}, results on Hopf monads established in \cite{MR2355605,MR2869176,MR2793022} can be translated into results on Hopf comonads. For example, under the assumption that $\mathcal{C}$ is rigid, it follows from \cite[Theorem 3.8]{MR2355605} that a monoidal comonad $T$ on $\mathcal{C}$ is a Hopf comonad if and only if $\mathcal{C}^T$ is rigid.

\subsection{Hopf modules}
\label{subsec:hopf-modules}

Let $\mathcal{C}$ be a finite tensor category, and let $T$ be a $k$-linear exact Hopf comonad on $\mathcal{C}$ with comultiplication $\delta$ and counit $\varepsilon$. Then the category $\mathcal{C}^T$ is also a finite tensor category such that the forgetful functor
\begin{equation}
  \label{eq:Hopf-monad-T-forget}
  U: \mathcal{C}^T \to \mathcal{C},
  \quad (V, \rho) \mapsto V
\end{equation}
preserves and reflects exact sequences. The free $T$-comodule functor
\begin{equation}
  \label{eq:Hopf-monad-T-free}
  R: \mathcal{C} \to \mathcal{C}^T,
  \quad V \mapsto (T(V), \delta_V)
\end{equation}
is right adjoint to $U$ with unit $\eta: \id_{\mathcal{C}^T} \to R U$ and counit $\epsilon: U R \to \id_{\mathcal{C}}$ given by
\begin{equation}
  \label{eq:Hopf-monad-T-unit-counit}
  \eta_{\mathbf{M}} = \rho_M
  \quad (\mathbf{M} = (M, \rho_M) \in \mathcal{C}^T)
  \quad \text{and} \quad
  \epsilon_V = \varepsilon_V
  \quad (V \in \mathcal{C}).
\end{equation}
By the definition of the tensor product of $\mathcal{C}^T$, the functor $U$ is strict monoidal. The functor $R$ is a $k$-linear exact monoidal functor with the same monoidal structure as $T$ ({\it i.e.}, $R_0 = T_0$ and $R_2 = T_2$), and the adjunction $U \dashv R$ is a {\em monoidal adjunction} in the sense that the unit and the counit of $U \dashv R$ are monoidal natural transformations.

The left Hopf operator $\mathbb{H}^{(\ell)}$ and the right Hopf operator $\mathbb{H}^{(r)}$ of the monoidal adjunction $U \dashv R$ are the natural transformations defined by
\begin{align*}
  \mathbb{H}^{(\ell)}_{X,\mathbf{M}}
  & := R_2(X, U(\mathbf{M})) \circ (\id_{R(X)} \otimes \rho_M)
    : R(X) \otimes \mathbf{M} \to R(X \otimes M)
  \quad \text{and} \\
  \mathbb{H}^{(r)}_{\mathbf{M},X}
  & := R_2(U(\mathbf{M}), X) \circ (\rho_M \otimes \id_{R(X)})
    : \mathbf{M} \otimes R(X) \to R(M \otimes X),
\end{align*}
respectively, for $X \in \mathcal{C}$ and $\mathbf{M} = (M, \rho_M) \in \mathcal{C}^T$ ({\it cf}. \cite[Subsection 2.8]{MR2793022}). Since $T$ is a Hopf comonad, the Hopf operators $\mathbb{H}^{(\ell)}$ and $\mathbb{H}^{(r)}$ are invertible \cite[Theorem 2.15]{MR2793022}. We note the following identities:

\begin{lemma}
  \label{lem:Hopf-ope-monoidal}
  For all $\mathbf{M} = (M, \rho_M)$, $\mathbf{N} = (N, \rho_N) \in \mathcal{C}^T$ and $X \in \mathcal{C}$, we have\textup{:}
  \begin{gather}
    \label{eq:Hopf-ope-1}
    \mathbb{H}^{(\ell)}_{X, \mathbf{M} \otimes \mathbf{N}}
    = \mathbb{H}^{(\ell)}_{X \otimes M, \mathbf{N}}
    \circ (\mathbb{H}^{(\ell)}_{X, \mathbf{M}} \otimes \id_{\mathbf{N}}^{}), \\
    \label{eq:Hopf-ope-2}
    \mathbb{H}^{(r)}_{\mathbf{M} \otimes \mathbf{N}, X}
    = \mathbb{H}^{(r)}_{\mathbf{M}, N \otimes X}
    \circ (\id_{\mathbf{M}}^{} \otimes \mathbb{H}^{(r)}_{\mathbf{N}, X}), \\
    \label{eq:Hopf-ope-3}
    \mathbb{H}^{(\ell)}_{X, \unitobj^T}
    = \id_{R(X)}^{}
    = \mathbb{H}^{(r)}_{\unitobj^T, X}, \\
    \label{eq:Hopf-ope-4}
    R(\varepsilon_X) \circ \mathbb{H}^{(\ell)}_{\unitobj, R(X)} = R_2(\unitobj, X),
    \quad
    R(\varepsilon_X) \circ \mathbb{H}^{(r)}_{R(X), \unitobj} = R_2(X, \unitobj).
  \end{gather}
\end{lemma}
\begin{proof}
  We prove~\eqref{eq:Hopf-ope-1}. By the definition of $\mathbb{H}^{(\ell)}$, we compute:
  \begin{align*}
    \mathbb{H}^{(\ell)}_{X, \mathbf{M} \otimes \mathbf{N}}
    & = T_2(X, M \otimes N) \circ (\id_{T(X)}^{} \otimes T_2(X,Y)(\rho_M \otimes \rho_N)) \\
    & = T_2(X \otimes M, N) \circ (T_2(X, M) \otimes \id_{T(N)})
      \circ (\id_{T(X)}^{} \otimes \rho_M \otimes \rho_N) \\
    & = \mathbb{H}^{(\ell)}_{X \otimes M, \mathbf{N}}
      \circ (\mathbb{H}^{(\ell)}_{X, \mathbf{M}} \otimes \id_{\mathbf{N}}^{}).
  \end{align*}
  Equation~\eqref{eq:Hopf-ope-2} is proved in a similar way. Equation~\eqref{eq:Hopf-ope-3} follows directly from the definition of monoidal functors. Equation~\eqref{eq:Hopf-ope-4} follows from the naturality of $R$ and the definition of comodules.
\end{proof}

The object $\mathbf{A} := R(\unitobj)$ is an algebra in $\mathcal{C}^T$ with multiplication $m := R_2(\unitobj, \unitobj)$ and unit $u := R_0$. We define the category $\mathcal{C}^T_T$ of {\em right Hopf modules over $T$} to be the category of right $\mathbf{A}$-modules in $\mathcal{C}^T$. By~\eqref{eq:Hopf-ope-1}--\eqref{eq:Hopf-ope-4}, the pair $(\mathbf{A}, \hat{\sigma})$, where
\begin{equation}
  \label{eq:adj-alg-half-br}
  \hat{\sigma}(\mathbf{M})
  = (\mathbb{H}^{(r)}_{\mathbf{M}, \unitobj})^{-1}
  \circ \mathbb{H}^{(\ell)}_{\unitobj, \mathbf{M}}:
  \mathbf{A} \otimes \mathbf{M} \to \mathbf{M} \otimes \mathbf{A}
  \quad (\mathbf{M} \in \mathcal{C}^T),
\end{equation}
is a commutative algebra in $\mathcal{Z}(\mathcal{C}^T)$ \cite[Theorem~6.6]{MR2793022}. Thus a right Hopf module $(\mathbf{M}, \mu: \mathbf{M} \otimes \mathbf{A} \to \mathbf{M})$ in $\mathcal{C}^T$ can be regarded as an $\mathbf{A}$-bimodule by defining the left action $\mu'$ of $\mathbf{A}$ on $\mathbf{M}$ by $\mu' = \mu \circ \hat{\sigma}(\mathbf{M})$.

The category $\mathcal{C}^T_T$ is a monoidal category as a monoidal full subcategory of the category of $\mathbf{A}$-bimodules in $\mathcal{C}^T$. We consider the functor
\begin{equation}
  \label{eq:Hopf-monad-T-FTHM}
  H: \mathcal{C} \to \mathcal{C}^T_T,
  \quad V \mapsto (R(V), R_2(V, \unitobj)).
\end{equation}
The functor $H$ is a monoidal functor with the structure morphisms defined as follows: The morphism $H_0$ is given by $H_0 = \id_{\mathbf{A}}$. To define the natural transformation $H_2$, we first note that the equation
\begin{equation*}
  R_2(\unitobj, V) \circ \hat{\sigma}(R(V)) = R_2(V, \unitobj)
  \quad (V \in \mathcal{C})
\end{equation*}
holds by \eqref{eq:Hopf-ope-4} ({\it cf}. \cite[Theorem~6.6]{MR2793022}). The left-hand side is just the left action of $\mathbf{A}$ on $H(V)$. Thus, for $V, W \in \mathcal{C}$, the tensor product of $H(V)$ and $H(W)$ over $\mathbf{A}$ is given by the coequalizer diagram
\begin{equation*}
  \xymatrix{
    H(V) \otimes \mathbf{A} \otimes H(W)
    \ar@<+.75ex>[rrr]^(.525){R_2(V, \unitobj) \otimes \id_{R(W)}}
    \ar@<-.75ex>[rrr]_(.525){\id_{R(V)} \otimes R_2(\unitobj, W)}
    & & & H(V) \otimes H(W) \ar[r]^(.475){q_{V,W}}
    & H(V) \otimes_{\mathbf{A}} H(W).
  }
\end{equation*}
By the universal property, we define $H_2(V, W): H(V) \otimes_{\mathbf{A}} H(W) \to H(V \otimes W)$ to be the unique morphism satisfying
\begin{equation*}
  H_2(V, W) \circ q_{V,W} = R_2(V, W).
\end{equation*}
By interpreting \cite[Section 6]{MR2793022} in our context, we have the following {\em fundamental theorem for Hopf modules}:

\begin{lemma}
  \label{lem:FTHM}
  If $T$ is conservative, then the monoidal functor $H: \mathcal{C} \to \mathcal{C}^T_T$ defined in the above is an equivalence of $k$-linear monoidal categories.
\end{lemma}

Here, a functor is said to be {\em conservative} if it reflects isomorphisms. We note that an exact functor between abelian categories is conservative if and only if it is faithful \cite[Lemma 5.7]{2014arXiv1402.3482S}.

\subsection{Hopf bimodules}
\label{subsec:hopf-bimodules}

We keep the notation as in Subsection~\ref{subsec:hopf-modules} and assume moreover that the Hopf comonad $T$ is conservative. We define the category ${}_T^{}\mathcal{C}_T^T$ of {\em Hopf bimodules} over $T$ to be the category of $\mathbf{A}$-bimodules in $\mathcal{C}^T$.

Since the forgetful functor $U: \mathcal{C}^T \to \mathcal{C}$ is strict monoidal, the category $\mathcal{C}$ is a left $\mathcal{C}^T$-module category by $\mathbf{M} \ogreaterthan V = U(\mathbf{M}) \otimes V$ ($\mathbf{M} \in \mathcal{C}^T$, $V \in \mathcal{C}$). Lemma~\ref{lem:Hopf-ope-monoidal} implies that the functors $R: \mathcal{C} \to \mathcal{C}^T$ and $H: \mathcal{C} \to \mathcal{C}^T_T$ are $\mathcal{C}^T$-module functors by the right Hopf operator $\mathbb{H}^{(r)}$. Now let $\mathbf{B}$ be an algebra in $\mathcal{C}^T$. Since $H$ is an equivalence of $\mathcal{C}^T$-module categories by Lemma~\ref{lem:FTHM}, it induces a category equivalence
\begin{equation*}
  H: {}_{\mathbf{B}}\mathcal{C} \to {}_{\mathbf{B}} (\mathcal{C}^T_T) = {}_{\mathbf{B}} (\mathcal{C}^T)_{\mathbf{A}}
  \quad (\text{$:=$ the category of $\mathbf{B}$-$\mathbf{A}$-bimodules
    in $\mathcal{C}^T$}).
\end{equation*}
If $\mathbf{B} = \mathbf{A}$, then ${}_{\mathbf{B}}\mathcal{C}$ is the category of $T(\unitobj)$-modules in $\mathcal{C}$ and ${}_{\mathbf{B}}(\mathcal{C}^T)_{\mathbf{A}}$ is the category of Hopf bimodules over $T$. Thus we have obtained the following {\em fundamental theorem for Hopf bimodules} over a Hopf comonad:

\begin{lemma}
  \label{lem:FTHBM}
  The functor $H$ induces an equivalence ${}_{T(\unitobj)} \mathcal{C} \approx {}_T^{} \mathcal{C}^T_T$ of categories.
\end{lemma}

\subsection{The central Hopf comonad}
\label{subsec:central-Hopf-comonad}

Let $\mathcal{C}$ be a finite tensor category. Then there is a $k$-linear faithful exact comonad $Z$ on $\mathcal{C}$ whose category of comodules can be identified with the Drinfeld center of $\mathcal{C}$. The Hopf comonad $Z$, which is called the {\em central Hopf comonad} in \cite{2015arXiv150401178S}, helps us to analyze the Drinfeld center.

We recall the definition of $Z$. For $V \in \mathcal{C}$, we define $Z(V)$ to be the end
\begin{equation*}
  Z(V) = \int_{X \in \mathcal{C}} X \otimes V \otimes X^*,
\end{equation*}
which indeed exists by the argument of \cite[Section 5]{MR1862634}. We denote by
\begin{equation}
  \label{eq:Hopf-comonad-Z-pi}
  \pi_V(X): Z(V) \to X \otimes V \otimes X^* \quad (X \in \mathcal{C})
\end{equation}
the universal dinatural transformation of the end $Z(V)$. By the parameter theorem for ends \cite[IX.7]{MR1712872}, the assignment $V \mapsto Z(V)$ extends to an endofunctor on $\mathcal{C}$ in such a way that $\pi_V(X)$ is natural in the variable $V$.

The functor $Z$ has a structure of a monoidal comonad given as follows: By using the universal property of the end, the monoidal structure
\begin{equation*}
  Z_0: \unitobj \to Z(\unitobj)
  \quad \text{and} \quad
  Z_2(V, W): Z(V) \otimes Z(W) \to Z(V \otimes W)
\end{equation*}
for $V, W \in \mathcal{C}$ are defined to be the unique morphisms such that the equations
\begin{gather}
  \label{eq:Hopf-comonad-Z-Z2}
  \pi_{V \otimes W}(X) \circ Z_2(V, W)
  = (\id_{X \otimes V} \otimes \eval_X \otimes \id_{W \otimes X^*}) \circ (\pi_V(X) \otimes \pi_W(X)), \\
  \label{eq:Hopf-comonad-Z-Z0}
  \pi_{\unitobj}(X) \circ Z_0
  = \coev_X
\end{gather}
hold for all $X \in \mathcal{C}$. To define the comultiplication of $Z$, we note that the tensor product of $\mathcal{C}$ is continuous in each variable. Thus, by the Fubini theorem for ends \cite[IX.8]{MR1712872}, the object $Z Z(V)$ is the end
\begin{equation*}
  Z Z(V) = \int_{X, Y \in \mathcal{C}} X \otimes Y \otimes V \otimes Y^* \otimes X^*
\end{equation*}
with the universal dinatural transformation given by
\begin{equation}
  \label{eq:Hopf-comonad-Z-pi2}
  \pi^{(2)}_{V}(X, Y)
  := (\id_{X} \otimes \pi_V(Y) \otimes \id_{X^*}) \circ \pi_{Z(V)}(X)
\end{equation}
for $X, Y \in \mathcal{C}$. By the universal property, we define $\delta_V: Z(V) \to Z Z(V)$ to be the unique morphism such that the equation
\begin{equation}
  \label{eq:Hopf-comonad-Z-delta}
  \pi_{V}^{(2)}(X, Y) \circ \delta_{V} = \pi_{V}(X \otimes Y)
\end{equation}
holds for all $X, Y \in \mathcal{C}$. Finally, we define $\varepsilon_V: Z(V) \to V$ for $V \in \mathcal{C}$ by
\begin{equation}
  \label{eq:Hopf-comonad-Z-eps}
  \varepsilon_V = \pi_{V}(\unitobj).
\end{equation}

The category $\mathcal{C}^Z$ of $Z$-comodules can be identified with $\mathcal{Z}(\mathcal{C})$ as follows: By the rigidity of $\mathcal{C}$ and basic properties of ends, we have isomorphisms
\begin{equation}
  \label{eq:Hom-V-ZV}
  \begin{aligned}
    \Hom_{\mathcal{C}}(V, Z(V))
    & \textstyle \cong \int_{X \in \mathcal{C}} \Hom_{\mathcal{C}}(V, X \otimes V \otimes X^*) \\
    & \textstyle \cong \int_{X \in \mathcal{C}} \Hom_{\mathcal{C}}(V \otimes X, X \otimes V) \\
    & \cong \Nat(V \otimes (-), (-) \otimes V)
  \end{aligned}
\end{equation}
for $V \in \mathcal{C}$. Let $\rho: V \to Z(V)$ be a morphism in $\mathcal{C}$, and let $\sigma: V \otimes (-) \to (-) \otimes V$ be the natural transformation corresponding to $\rho$ via~\eqref{eq:Hom-V-ZV}. Explicitly,
\begin{equation}
  \label{eq:def-sigma-from-rho}
  \sigma(X) = (\id_X \otimes \id_V \otimes \eval_X) \circ (\pi_V(X) \rho \otimes \id_X)
  \quad (X \in \mathcal{C}).
\end{equation}
One can prove that the pair $(V, \rho)$ is a $Z$-comodule if and only if the pair $(V, \sigma)$ is an object of $\mathcal{Z}(\mathcal{C})$. Hence we obtain a one-to-one correspondence between $\mathrm{Obj}(\mathcal{Z}(\mathcal{C}))$ and $\mathrm{Obj}(\mathcal{C}^Z)$, which, in fact, gives rise to an isomorphism $\mathcal{Z}(\mathcal{C}) \cong \mathcal{C}^Z$ of monoidal categories ({\it cf}. \cite{MR2342829,MR2869176}).

Since $\mathcal{Z}(\mathcal{C})$ is rigid, so is $\mathcal{C}^Z$, and hence $Z$ is a Hopf comonad. Now we define the functors $U$, $R$ and $H$ by \eqref{eq:Hopf-monad-T-forget}, \eqref{eq:Hopf-monad-T-free} and~\eqref{eq:Hopf-monad-T-FTHM} with $T = Z$, respectively. By the results of \cite{2014arXiv1402.3482S} on the adjunction $U \dashv R$, we see that $Z$ is $k$-linear, exact and faithful. Thus, by the fundamental theorem for Hopf modules (Lemma~\ref{lem:FTHM}), the functor $H: \mathcal{C} \to \mathcal{C}^Z_Z$ is an equivalence of monoidal categories. By Lemma~\ref{lem:FTHBM}, we also have an equivalence ${}_A \mathcal{C} \approx {}_Z^{} \mathcal{C}^Z_Z$, where $A = Z(\unitobj)$.

\begin{remark}
  The monoidal structure of $\mathcal{C}^Z_Z$ is defined by using the natural isomorphism $\hat{\sigma}$ given by \eqref{eq:adj-alg-half-br}. In terms not of the central Hopf comonad, but of the half-braiding, the natural isomorphism $\hat{\sigma}$ is given as follows:
  \begin{equation}
    \label{eq:rem-sigma-hat-1}
    \hat{\sigma}(\mathbf{M}) = \sigma_{M}(Z(\unitobj))^{-1}
    \quad (\mathbf{M} = (M, \sigma_{M}) \in \mathcal{Z}(\mathcal{C})).
  \end{equation}
  By the definition of $\hat{\sigma}$, equation~\eqref{eq:rem-sigma-hat-1} is equivalent to
  \begin{equation}
    \label{eq:rem-sigma-hat-2}
    \mathbb{H}^{(\ell)}_{\unitobj, \mathbf{M}} \circ \sigma_{M}(Z(\unitobj))
    = \mathbb{H}^{(r)}_{\mathbf{M}, \unitobj}.
  \end{equation}
  We verify~\eqref{eq:rem-sigma-hat-2} instead of \eqref{eq:rem-sigma-hat-1}. Let $\rho_M: M \to Z(M)$ be the coaction of $Z$ on $M$ corresponding to the half-braiding $\sigma_M$. By \eqref{eq:def-sigma-from-rho}, we have
  \begin{equation*}
    \pi_M(X) \circ \rho_M = (\sigma_M(X) \otimes \id_{X^*}) \circ (\id_M \otimes \coev_X)
  \end{equation*}
  for all $X \in \mathcal{C}$. Thus we compute
  \begin{align*}
    \pi_{M}(X) \circ \mathbb{H}^{(\ell)}_{\unitobj, \mathbf{M}}
    & = \pi_M(X) \circ Z_2(\unitobj, M) \circ (\id_{Z(\unitobj)} \otimes \rho_M) \\
    & = (\id_X \otimes \eval_X \otimes \id_M \otimes \id_{X^*})
      \circ (\pi_{\unitobj}(X) \otimes \pi_{M}(X) \rho_M) \\
    & = (\id_X \otimes \eval_X \otimes \id_M \otimes \id_{X^*}) \\
    & \qquad \circ (\pi_{\unitobj}(X) \otimes \sigma_M(X) \otimes \id_{X^*})
      \circ (\id_{Z(\unitobj)} \otimes \id_M \otimes \coev_X), \\
    \pi_{M}(X) \circ \mathbb{H}^{(r)}_{\mathbf{M}, \unitobj}
    & = \pi_M(X) \circ Z_2(M, \unitobj) \circ (\rho_M \otimes \id_{Z(\unitobj)}) \\
    & = (\id_X \otimes \id_M \otimes \eval_X \otimes \id_{X^*})
      \circ (\pi_{M}(X) \rho_M \otimes \pi_{\unitobj}(X)) \\
    & = (\id_X \otimes \id_M \otimes \eval_X \otimes \id_{X^*}) \\
    & \qquad \circ (\sigma_M(X) \otimes \id_{X^*} \otimes \pi_{\unitobj}(X))
      \circ (\id_M \otimes \coev_X \otimes \id_{Z(\unitobj)})
  \end{align*}
  by~\eqref{eq:Hopf-comonad-Z-Z2} and~\eqref{eq:def-sigma-from-rho}. Now~\eqref{eq:rem-sigma-hat-2} is proved by the universal property of $Z(M)$ and the following graphical calculus:
  \newcommand{\xylabel}[1]{\scriptstyle\mathstrut\smash{#1}}
  \knotholesize{1em}
  \begin{align*}
    \pi_M(X) \circ \mathbb{H}^{(r)}_{\mathbf{M}, \unitobj}
    & = {\xy /r1.25em/:
      (0,1.5)="P1" *+!D{\xylabel{M}},
      "P1"; p-(0,1) **\dir{-} ?> \pushrect{1}{1},
      \vtwist~{s0}{s1}{s2}{s3},
      s2; p-(0,1) **\dir{-} ?>="Q1" *+!U{\xylabel{X}},
      s3; p-(0,1) **\dir{-} ?>="Q2" *+!U{\xylabel{M}},
      s1; p+(1.25,0) \mycap{+.75}
      ?>; p-(0,1) **\dir{-}
      ?>; p+(1.25,0) \mycap{-.75}
      ?>="B1";
      p+(1,0)="B2"; \makecoord{p}{"Q1"}
      **\dir{-} ?> *+!U{\xylabel{X^*}},
      "B1"; "B2" **\dir{} ?(.5)
      *+!D{\makebox[2pc]{$\xylabel{\pi_{\unitobj}(X)}$}} *\frm{-}="BX1",
      "BX1"!U; \makecoord{p}{"P1"} **\dir{-} ?> *+!D{\xylabel{Z(\unitobj)}},
      \endxy}
      = {\xy /r1.25em/:
      (0,1.5)="P1" *+!D{\xylabel{M}},
      "P1"; p-(0,1.5) **\dir{-} ?> \pushrect{1.5}{1},
      \vtwist~{s0}{s1}{s2}{s3},
      s2; p-(0,.5) **\dir{-} ?>="Q1" *+!U{\xylabel{X}},
      s3; p-(0,.5) **\dir{-} ?>="Q2" *+!U{\xylabel{M}},
      s1="B1";
      p+(1,0)="B2"; \makecoord{p}{"Q1"}
      **\dir{-} ?> *+!U{\xylabel{X^*}},
      "B1"; "B2" **\dir{} ?(.5)
      *+!D{\makebox[2pc]{$\xylabel{\pi_{\unitobj}(X)}$}} *\frm{-}="BX1",
      "BX1"!U; \makecoord{p}{"P1"} **\dir{-} ?> *+!D{\xylabel{Z(\unitobj)}},
      \endxy}
      = {\xy /r1.25em/:
      (0,1.5)="P1" \pushrect{1.75}{1.5},
      \vtwist~{s0}{s1}{s2}{s3},
      s0 *+!D{\xylabel{M}},
      s1 *+!D{\xylabel{Z(\unitobj)}},
      s2 *+!U{\makebox[2pc]{$\xylabel{\pi_{\unitobj}(X)}$}} *\frm{-}="BX1",
      "BX1"!D!L(.66); \makecoord{p}{"P1"-(0,3)}="Q1" **\dir{-}
      ?> *+!U{\xylabel{X}},
      "BX1"!D!R(.66)="T1",
      s3 \pushrect{1}{1},
      \vtwist~{s0}{s1}{s2}{s3},
      "T1"; s2 \mycap{-.75},
      s3; \makecoord{p}{"Q1"} **\dir{-}
      ?>  *+!U{\xylabel{M}},
      s1; p+(1.5,0) \mycap{+.75}
      ?>; \makecoord{p}{"Q1"} **\dir{-}
      ?>  *+!U{\xylabel{X^*}},
    \endxy} \\
    & = \pi_{M}(X) \circ \mathbb{H}^{(\ell)}_{\unitobj, \mathbf{M}} \circ \sigma_{M}(Z(\unitobj)).
  \end{align*}
  It is well-known that the algebra $\mathbf{A} := R(\unitobj)$ is a commutative algebra in $\mathcal{Z}(\mathcal{C})$; see, {\it e.g.}, \cite[Lemma 3.5]{MR3039775}. This fact can be proved by \eqref{eq:rem-sigma-hat-1} and a general result on Hopf comonad as follows: Let $\mathbf{c}$ and $\hat{\mathbf{c}}$ be the braiding of $\mathcal{Z}(\mathcal{C})$ and $\mathcal{Z}(\mathcal{Z}(\mathcal{C}))$, respectively. We write $\mathbf{A} = (A, \sigma_A) \in \mathcal{Z}(\mathcal{C})$. By the theory of Hopf comonads, we see that $\hat{\mathbf{A}} := (\mathbf{A}, \hat{\sigma})$ is a commutative algebra in $\mathcal{Z}(\mathcal{Z}(\mathcal{C}))$. Thus, by \eqref{eq:rem-sigma-hat-1},
  \begin{equation*}
    m \circ \mathbf{c}_{\mathbf{A}, \mathbf{A}} = m \circ \sigma_A(A) = m \circ \hat{\sigma}(\mathbf{A})^{-1} = m \circ \hat{\mathbf{c}}_{\hat{\mathbf{A}}, \hat{\mathbf{A}}}^{\,-1} = m.
  \end{equation*}
\end{remark}

\subsection{The adjoint algebra}

Keep the notations as in Subsection~\ref{subsec:central-Hopf-comonad}. We call the algebra $A := U(\mathbf{A})$ the {\em adjoint algebra} of $\mathcal{C}$ as it generalizes the adjoint representation of a Hopf algebra (see \cite{2015arXiv150401178S}). We give some useful results about the algebras $A \in \mathcal{C}$ and $\mathbf{A} \in \mathcal{Z}(\mathcal{C})$.

For an object $V \in \mathcal{C}$, we define
\begin{equation}
  \label{eq:can-act-1}
  \canact_{V} = (\id_V \otimes \eval_V) \circ (\pi_{\unitobj}(V) \otimes \id_V)
  : A \otimes V \to V.
\end{equation}
Then the pair $(V, \canact_{V})$ is a left $A$-module in $\mathcal{C}$ (see \cite{2015arXiv150401178S}), and thus we call $\canact_V$ the {\em canonical action} of $A$ on $V$. This notion is used to establish the following category equivalence:

\begin{lemma}
  \label{lem:adj-alg-mod}
  The functor defined by
  \begin{equation}
    \label{eq:adj-alg-mod-equiv}
    \mathcal{C} \boxtimes \mathcal{C} \to {}_{A} \mathcal{C},
    \quad V \boxtimes W \mapsto (V \otimes W, \canact_V \otimes \id_W)
  \end{equation}
  is an equivalence of categories.
\end{lemma}
\begin{proof}
  It is known that the coend $L = \int^{X \in \mathcal{C}} X \otimes X^*$ has a canonical structure of a coalgebra in $\mathcal{C}$. As explained in \cite{2015arXiv150401178S}, the algebra $A$ is dual to the coalgebra $L$. The category of $L$-comodules in $\mathcal{C}$ has been studied by Lyubashenko in \cite{MR1625495}. This lemma is proved by rephrasing \cite[Corollary 2.7.2]{MR1625495}.
\end{proof}

\begin{remark}
  This lemma can also be proved by using basic results on exact module categories established in \cite{MR2119143}. We regard $\mathcal{C}$ as a left $\mathcal{Z}(\mathcal{C})$-module category through the forgetful functor $U: \mathcal{Z}(\mathcal{C}) \to \mathcal{C}$. Then, since $\mathcal{C} \approx \mathcal{Z}(\mathcal{C})_{\mathbf{A}}$ as left $\mathcal{Z}(\mathcal{C})$-module categories, we have equivalences
  \begin{equation*}
    {}_{\mathbf{A}}\mathcal{Z}(\mathcal{C})_{\mathbf{A}}
    \approx \text{(the category of $k$-linear $\mathcal{Z}(\mathcal{C})$-module functors $\mathcal{C} \to \mathcal{C}$)}
    \approx \mathcal{C} \boxtimes \mathcal{C}
  \end{equation*}
  of categories \cite{MR2119143}. Since $\mathcal{C}_A \approx {}_Z^{} \mathcal{C}_Z^Z = {}_{\mathbf{A}}\mathcal{Z}(\mathcal{C})_{\mathbf{A}} $ by the fundamental theorem of Hopf bimodules, we have $\mathcal{C}_A \approx \mathcal{C} \boxtimes \mathcal{C}$. The resulting equivalence is in fact given as stated in the above lemma.
\end{remark}

Given an invertible object $\beta \in \mathcal{C}$, we define the {\em character} of $\beta$ by
\begin{equation}
  \label{eq:def-character}
  \intch_{\beta} := \coev_{\beta}^{-1} \circ \pi_{\unitobj}(\beta) \in \Hom_{\mathcal{C}}(A, \unitobj)
\end{equation}
({\it cf}. \cite{2015arXiv150401178S}). By the dinaturality of $\pi_{\unitobj}: A \to X \otimes X^*$, we see that the character of $\beta$ depends only on the isomorphism class of $\beta$. Moreover, the following lemma shows that the character of an invertible object is multiplicative:

\begin{lemma}
  $\intch_{\beta}: A \to \unitobj$ is a morphism of algebras in $\mathcal{C}$.
\end{lemma}
\begin{proof}
  An object of the form $X \otimes X^*$ for some $X \in \mathcal{C}$ is an algebra in $\mathcal{C}$ with multiplication $\id_X \otimes \eval_X \otimes \id_{X^*}$ and unit $\coev_X$. It is easy to see that $\pi_{\unitobj}(X)$ and $\coev_X$ are morphisms of algebras in $\mathcal{C}$. Thus $\intch_{\beta}$ is a morphism of algebras in $\mathcal{C}$ as the composition of such morphisms.
\end{proof}

For two algebras $P$ and $Q$ in $\mathcal{C}$, we denote by $\Alg_{\mathcal{C}}(P, Q)$ the set of algebra morphisms from $P$ to $Q$. We also denote by $\Inv(\mathcal{C})$ the set of isomorphism classes of invertible objects of $\mathcal{C}$. By the above argument, we obtain the map
\begin{equation}
  \label{eq:character-map}
  \intch: \Inv(\mathcal{C}) \to \Alg_{\mathcal{C}}(A, \unitobj),
  \quad [\beta] \mapsto \intch_{\beta}.
\end{equation}

\begin{lemma}
  \label{lem:Alg-C-A-1}
  This map is bijective.
\end{lemma}
\begin{proof}
  We first prove that the map~\eqref{eq:character-map} is surjective. Let $\chi: A \to \unitobj$ be a morphism of algebras in $\mathcal{C}$. Since $\unitobj \in \mathcal{C}$ has no proper subobject, $(\unitobj, \chi)$ is a simple object of ${}_A \mathcal{C}$. By the representation theory of finite-dimensional algebras, a simple object of $\mathcal{C} \boxtimes \mathcal{C}$ is of the form $V \boxtimes W$ for some simple objects $V$ and $W$ of $\mathcal{C}$. Thus, by Lemma~\ref{lem:adj-alg-mod}, there are simple objects $V$ and $W$ of $\mathcal{C}$ such that
  \begin{equation*}
    (\unitobj, \chi) \cong (V \otimes W, \canact_{V} \otimes \id_{W})
  \end{equation*}
  as left $A$-modules. In particular, $V \otimes W \cong \unitobj$. Thus we may assume that $\beta := V$ is an invertible object and $W = \beta^*$. Since $\Hom_{\mathcal{C}}(\unitobj, \beta \otimes \beta^*)$ is the one-dimensional vector space spanned by $\coev_{\beta}$, we see that
  \begin{equation*}
    \coev_{\beta}: (\unitobj, \chi) \to (\beta \otimes \beta^*, \canact_{\beta} \otimes \id_{\beta^*})
  \end{equation*}
  must be a morphism of left $A$-modules. This means that the equation
  \begin{equation*}
    \coev_{\beta} \circ \chi
    = (\canact_{\beta} \otimes \id_{\beta^*}) \circ (\id_{A} \otimes \coev_{\beta})
  \end{equation*}
  holds. Therefore $\chi = \intch_{\beta}$. This proves the surjectivity. For the later discussion, we also note that the above argument shows that $(\unitobj, \intch_{\beta})$ corresponds to $\beta \boxtimes \beta^*$ via the equivalence ${}_{A}\mathcal{C} \approx \mathcal{C} \boxtimes \mathcal{C}$ given in Lemma~\ref{lem:adj-alg-mod} 

  Now we check the injectivity of \eqref{eq:character-map}. Let $\alpha$ and $\beta$ be invertible objects of $\mathcal{C}$ such that $\intch_{\alpha} = \intch_{\beta}$. Then $(\unitobj, \intch_{\alpha}) = (\unitobj, \intch_{\beta})$ as left $A$-modules. By Lemma~\ref{lem:adj-alg-mod} and the above argument, we have $\alpha \boxtimes \alpha^* \cong \beta \boxtimes \beta^*$ in $\mathcal{C} \boxtimes \mathcal{C}$. Since
  \begin{equation*}
    \Hom_{\mathcal{C} \boxtimes \mathcal{C}}(\alpha \boxtimes \alpha^*, \beta \boxtimes \beta^*)
    \cong \Hom_{\mathcal{C}}(\alpha, \beta) \otimes_k \Hom_{\mathcal{C}}(\alpha^*, \beta^*),
  \end{equation*}
  we have $\alpha \cong \beta$. Hence the map \eqref{eq:character-map} is injective.
\end{proof}

The vector space $\mathrm{CF}(\mathcal{C}) := \Hom_{\mathcal{C}}(A, \unitobj)$ is called the {\em space of class functions} in \cite{2015arXiv150401178S} as it generalizes the space of class functions on a finite group. This is an algebra over $k$ with respect to the convolution product
\begin{equation*}
  f \star g = f \circ Z(g) \circ \delta_{\unitobj}
  \quad (f, g \in \mathrm{CF}(\mathcal{C})),
\end{equation*}
and the adjunction isomorphism
\begin{equation*}
  \mathrm{CF}(\mathcal{C})
  \to \Hom_{\mathcal{Z}(\mathcal{C})}(\mathbf{A}, \mathbf{A}),
  \quad f \mapsto \widetilde{f} := Z(f) \circ \delta_{\unitobj}
\end{equation*}
is in fact an isomorphism of algebras over $k$ \cite{2015arXiv150401178S}.

\begin{lemma}
  \label{lem:Alg-C-A-1-Alg-ZC-A-A}
  The above map restricts to the bijection
  \begin{equation*}
    \Alg_{\mathcal{C}}(A, \unitobj)
    \to \Alg_{\mathcal{Z}(\mathcal{C})}(\mathbf{A}, \mathbf{A}),
    \quad f \mapsto \widetilde{f}.
  \end{equation*}
\end{lemma}
\begin{proof}
  We recall that $\mathbf{A}$ is an algebra with multiplication $m = Z_2(\unitobj, \unitobj)$ and unit $u = Z_0$. Let $f: A \to \unitobj$ be a morphism in $\mathcal{C}$, and let $\widetilde{f}$ be the morphism in $\mathcal{Z}(\mathcal{C})$ defined in the above. If $f$ is a morphism of algebras in $\mathcal{C}$, then we have
  \begin{align*}
    m \circ (\widetilde{f} \otimes \widetilde{f})
    & = Z_2(\unitobj, \unitobj) \circ (Z(f) \otimes Z(f)) \circ (\delta_{\unitobj} \otimes \delta_{\unitobj}) \\
    & = Z(f \otimes f) \circ Z_2(Z(\unitobj), Z(\unitobj)) \circ (\delta_{\unitobj} \otimes \delta_{\unitobj}) \\
    & = Z(f \circ m) \circ Z_2(Z(\unitobj), Z(\unitobj)) \circ (\delta_{\unitobj} \otimes \delta_{\unitobj}) \\
    & = Z(f) \circ Z(Z_2(\unitobj, \unitobj)) \circ Z_2(Z(\unitobj), Z(\unitobj)) \circ (\delta_{\unitobj} \otimes \delta_{\unitobj}) \\
    & = Z(f) \circ \delta_{\unitobj \otimes \unitobj} \circ Z_2(\unitobj, \unitobj)
    = \widetilde{f} \circ m.
  \end{align*}
  Here, the second equality follows from the naturality of $Z_2$, the third equality from that $f$ is a morphism of algebras, and the fifth from the fact that $\delta$ is a monoidal natural transformation. We also have
  \begin{equation*}
    \widetilde{f} \circ u
    = Z(f) \circ \delta_{\unitobj} \circ Z_0
    = Z(f) \circ Z(Z_0) \circ Z_0
    = Z(f \circ u) \circ u
    = Z(\id_{\unitobj}) \circ u = u.
  \end{equation*}
  Thus $\widetilde{f}$ is a morphism of algebras in $\mathcal{Z}(\mathcal{C})$. If, conversely, $\widetilde{f}$ is a morphism of algebras in $\mathcal{Z}(\mathcal{C})$, then we have
  \begin{gather*}
    f \circ m
    = \varepsilon_{\unitobj} \circ \widetilde{f} \circ m
    = \varepsilon_{\unitobj} \circ m \circ (\widetilde{f} \otimes \widetilde{f})
    = (\varepsilon_{\unitobj} \otimes \varepsilon_{\unitobj}) \circ (\widetilde{f} \otimes \widetilde{f})
    = f \otimes f, \\
    f \circ u
    = \varepsilon_{\unitobj} \circ \widetilde{f} \circ u
    = \varepsilon_{\unitobj} \circ u
    = \varepsilon_{\unitobj} \circ Z_0
    = Z_0 = u.
  \end{gather*}
  Thus $f$ is a morphism of algebras. In summary, $f$ is a morphism of algebras in $\mathcal{C}$ if and only if $\widetilde{f}$ is a morphism of algebras in $\mathcal{Z}(\mathcal{C})$. The proof is completed.
\end{proof}

The following lemma is one of key observations for the proof of the main result of this paper.

\begin{lemma}
  \label{lem:Alg-ZC-A-A}
  We have a bijective map
  \begin{equation*}
    \Inv(\mathcal{C})
    \to \Alg_{\mathcal{Z}(\mathcal{C})}(\mathbf{A}, \mathbf{A}),
    \quad [\beta] \mapsto Z(\intch_{\beta}) \circ \delta_{\unitobj}.
  \end{equation*}
\end{lemma}
\begin{proof}
  Compose the bijections obtained by Lemmas~\ref{lem:Alg-C-A-1} and~\ref{lem:Alg-C-A-1-Alg-ZC-A-A}.
\end{proof}

\begin{remark}
  The sets $\Inv(\mathcal{C})$ and $\Alg_{\mathcal{Z}(\mathcal{C})}(\mathbf{A}, \mathbf{A})$ are monoids with respect to the tensor product and the composition of morphisms, respectively. One can check that the set $\Alg_{\mathcal{C}}(A, \unitobj)$ is a sub\-monoid of $\mathrm{CF}(\mathcal{C})$. In this section, we have obtained the following two bijections:
  \begin{equation*}
    \Inv(\mathcal{C})
    \xrightarrow{\quad \text{Lemma~\ref{lem:Alg-C-A-1}} \quad}
    \Alg_{\mathcal{C}}(A, \unitobj)
    \xrightarrow{\quad \text{Lemma~\ref{lem:Alg-C-A-1-Alg-ZC-A-A}} \quad}
    \Alg_{\mathcal{Z}(\mathcal{C})}(\mathbf{A}, \mathbf{A}).
  \end{equation*}
  These bijections are in fact isomorphisms of monoids. Since $\Inv(\mathcal{C})$ is in fact a group (with the inverse given by $[\beta]^{-1} = [\beta^*]$), the monoids $\Alg_{\mathcal{C}}(A, \unitobj)$ and $\Alg_{\mathcal{Z}(\mathcal{C})}(\mathbf{A}, \mathbf{A})$ are also groups.
\end{remark}

\section{Pivotal structures of the Drinfeld center}
\label{sec:main-result}

\subsection{Main result}
\label{subsec:main-result}

Let $\mathcal{C}$ be a finite tensor category, and let $\cAut_{\otimes}(\mathcal{C})$ be the category of tensor autoequivalences of $\mathcal{C}$. Given $F \in \cAut_{\otimes}(\mathcal{C})$, we denote by
\begin{equation*}
  \widetilde{F}: \mathcal{Z}(\mathcal{C}) \to \mathcal{Z}(\mathcal{C})
\end{equation*}
the braided tensor autoequivalence induced by $F$. If we write $\widetilde{F}(\mathbf{V}) = (F(V), \sigma_V^F)$ for $\mathbf{V} = (V, \sigma_V) \in \mathcal{Z}(\mathcal{C})$, then the half-braiding $\sigma_V^F: F(V) \otimes (-) \to (-) \otimes F(V)$ is determined by the property that the equation
\begin{equation}
  \label{eq:def-half-br-F-tilde}
  F_2(X, V) \circ \sigma^F_V(F(X)) = F(\sigma(X)) \circ F_2(V, X)
\end{equation}
holds for all objects $X \in \mathcal{C}$. The monoidal structure of $\widetilde{F}$ is same as $F$.

Let $\mathscr{J}(F, G)$ be the class of all pairs $(\beta, j)$ consisting of an invertible object $\beta \in \mathcal{C}$ and a monoidal natural transformation $j \in \Nat_{\otimes}(I^{\beta} F, G)$. Given two elements $\boldsymbol{\beta} = (\beta, j)$ and $\boldsymbol{\beta}' = (\beta', j')$ of $\mathscr{J}(F, G)$, we write $\boldsymbol{\beta} \sim \boldsymbol{\beta}'$ if there exists an isomorphism $c: \beta \to \beta'$ in $\mathcal{C}$ such that the equation
\begin{equation}
  \label{eq:def-J-F-G-equiv}
  j'_X = j_X \circ (c^{-1} \otimes \id_{F(X)} \otimes c^*)
\end{equation}
holds for all objects $X \in \mathcal{C}$. It is easy to see that $\sim$ is an equivalence relation on the class $\mathscr{J}(F, G)$. We set $J(F, G) = \mathscr{J}(F, G) / \mathord{\sim}$. Now we are ready to state the main result of this section as follows:

\begin{theorem}
  \label{thm:main-thm}
  Given an element $\boldsymbol{\beta} = (\beta, j)$ of $\mathscr{J}(F, G)$, we define
  \begin{equation}
    \label{eq:def-Phi-1}
    \Phi(\boldsymbol{\beta})_{\mathbf{V}}: \widetilde{F}(\mathbf{V}) \to \widetilde{G}(\mathbf{V})
    \quad (\mathbf{V} \in \mathcal{Z}(\mathcal{C}))
  \end{equation}
  to be the composition
  \begin{equation}
    \label{eq:def-Phi-2}
    \newcommand{\xarr}[1]{\xrightarrow{\ #1 \ }}
    F(V) \xarr{\id \otimes \coev_{\beta}}
    F(V) \otimes \beta \otimes \beta^*
    \xarr{\sigma^F_V(\beta) \otimes \id}
    \beta \otimes F(V) \otimes \beta^*
    \xarr{j_{V}}
    G(V)
  \end{equation}
  for $\mathbf{V} = (V, \sigma_V) \in \mathcal{Z}(\mathcal{C})$ with $\widetilde{F}(\mathbf{V}) = (F(V), \sigma_V^F)$. Then the map
  \begin{equation}
    \label{eq:Phi}
    \Phi_{F,G}: J(F, G)
    \to \Nat_{\otimes}(\widetilde{F}, \widetilde{G}),
    \quad [\boldsymbol{\beta}] \mapsto \Phi(\boldsymbol{\beta})
  \end{equation}
  is a well-defined bijection.
\end{theorem}

If $G = (-)^{**}$ is the double dual functor on $\mathcal{C}$, then $\widetilde{G}$ is precisely the double dual functor on $\mathcal{Z}(\mathcal{C})$ by Remark~\ref{rem:Dri-cen-double-dual}. Thus we obtain Theorem~\ref{thm:classify-piv-str} as the special case where $F = \id_{\mathcal{C}}$ and $G = (-)^{**}$. More applications of this theorem will be given in Section~\ref{sec:applications}. The rest of this section is devoted to prove Theorem~\ref{thm:main-thm}.

\begin{notation*}
  Till the end of this section, we fix a finite tensor category $\mathcal{C}$.  We identify $\mathcal{Z}(\mathcal{C})$ with the category $\mathcal{C}^Z$ of comodules over the central Hopf comonad $Z$, and then define $U$ and $R$ by~\eqref{eq:Hopf-monad-T-forget} and~\eqref{eq:Hopf-monad-T-free} with $T = Z$, respectively. The symbols
  \begin{gather*}
    \pi_{V}(X): Z(V) \to X \otimes V \otimes X^*,
    \quad \delta_V: Z(V) \to Z Z(V),
    \quad \varepsilon_V: Z(V) \to V, \\
    Z_2(V,W): Z(V) \otimes Z(W) \to Z(V \otimes W)
    \quad \text{and} \quad
    Z_0: \unitobj \to Z(\unitobj)
  \end{gather*}
  have the same meaning as in Section~\ref{sec:central-Hopf-comonad}; see \eqref{eq:Hopf-comonad-Z-pi}--\eqref{eq:Hopf-comonad-Z-eps}. We denote by $\eta: \id_{\mathcal{Z}(\mathcal{C})} \to R U$ the unit of the adjunction $U \dashv R$ given by \eqref{eq:Hopf-monad-T-unit-counit} with $T = Z$.
\end{notation*}

\subsection{Outline of this section}
\label{subsec:outline-main-result}

This section is outlined as follows: In \S\ref{subsec:comon-desc-G-tilde}, we express the equivalence $\widetilde{G}$ induced by $G \in \cAut_{\otimes}(\mathcal{C})$ in terms of the central Hopf comonad. In \S\ref{subsec:well-def-Phi}, we show that the map $\Phi_{F,G}$ is well-defined, that is, $\Phi(\boldsymbol{\beta})$ defined by~\eqref{eq:def-Phi-2} is a monoidal natural transformation depending only on the equivalence class of $\boldsymbol{\beta}$. In \S\ref{subsec:reduction}, we explain that it is enough to consider the case where $F = \id_{\mathcal{C}}$ to prove Theorem~\ref{thm:main-thm}. Some technical lemmas are prepared in \S\ref{subsec:technical-lemmas}. Now, in view of the discussion of \S\ref{subsec:reduction}, we fix a tensor autoequivalence $G$ on $\mathcal{C}$ and consider the map $\Phi_G := \Phi_{\id_{\mathcal{C}}, G}$. In \S\ref{subsec:Phi-inv}, we construct the map $\Psi_G: \Nat(\id_{\mathcal{Z}(\mathcal{C})}, \widetilde{G}) \to J(\id_{\mathcal{C}}, G)$ by Hopf comonadic techniques introduced in Section~\ref{sec:central-Hopf-comonad}. We then show that $\Psi_G \Phi_G$ and $\Phi_G \Psi_G$ are the identity maps in Subsections \ref{subsec:compute-Psi-Phi} and~\ref{subsec:compute-Phi-Psi}, respectively, by using some technical lemmas given in \S\ref{subsec:technical-lemmas}.

\subsection{Comonadic description of $\widetilde{G}$}
\label{subsec:comon-desc-G-tilde}

Let $G$ be a tensor autoequivalence of $\mathcal{C}$. We first describe the braided tensor autoequivalence $\widetilde{G}$ on $\mathcal{Z}(\mathcal{C})$ in terms of the central Hopf comonad $Z$. We recall that there is the natural transformation
\begin{equation*}
  \gamma_X^G: G(X^*) \to G(X)^* \quad (X \in \mathcal{C}),
\end{equation*}
which we call the duality transformation (see Section \ref{sec:prelim}). For an integer $n \ge 2$ and objects $X_1, \dotsc, X_n \in \mathcal{C}$, we denote by
\begin{equation*}
  G_n(X_1, \dotsc, X_n): G(X_1) \otimes \dotsb \otimes G(X_n) \to G(X_1 \otimes \dotsb \otimes X_n)
\end{equation*}
the morphism constructed by an iterative use of $G_2$. We fix an object $V \in \mathcal{C}$. Then, by Lemma~\ref{lem:end-equiv-twist}, we can define a morphism $\xi^G_V: G Z(V) \to Z G(V)$ in $\mathcal{C}$ to be the unique morphism such that the diagram
\begin{equation*}
  \xymatrix{
    G Z(V)
    \ar[d]_{\xi^G_V}
    \ar[rrr]^(.375){G(\pi_{V}(X))}
    & & &
    G(X \otimes V \otimes X^*)
    \ar[d]^{G_{3}'(V, X)} \\
    Z G(V)
    \ar[rrr]_(.375){\pi_{G(V)}(G(X))}
    & & &
    G(X) \otimes G(V) \otimes G(X)^*
  }
\end{equation*}
commutes for all $X \in \mathcal{C}$, where
\begin{equation}
  \label{eq:xi-G-def-can-iso}
  G_3'(V, X) := (\id_{G(X)} \otimes \id_{G(V)} \otimes \gamma_X^F) \circ G_{3}(X, V, X^*)^{-1}.
\end{equation}
Since the morphisms $G(\pi_V(X))$, $\pi_{G(V)}(G(X))$ and $G_3'(V, X)$ are natural in the variable $V$, the morphism $\xi^G_V$ is also natural in $V$. Since $G_3'(V,X)$ is invertible, we see that $\xi^G$ is in fact a natural isomorphism.

\begin{lemma}
  \label{lem:G-tilde-Z-comod}
  If we regard the autoequivalence $\widetilde{G}: \mathcal{Z}(\mathcal{C}) \to \mathcal{Z}(\mathcal{C})$ as an autoequivalence on the category $\mathcal{C}^Z$ of $Z$-comodules, then it is expressed by
  \begin{equation*}
    \widetilde{G}(\mathbf{V}) = (G(V), \xi^{G}_V G(\rho_V))
  \end{equation*}
  for all $Z$-comodules $\mathbf{V} = (V, \rho_V) \in \mathcal{C}^Z$.
\end{lemma}
\begin{proof}
  Let $\mathbf{V} = (V, \rho)$ be a $Z$-comodule, and let
  \begin{equation*}
    \sigma_V(X): V \otimes X \to V \otimes X
    \quad \text{and} \quad
    \sigma_V'(X): G(V) \otimes X \to X \otimes G(V)
    \quad (X \in \mathcal{C})
  \end{equation*}
  be the natural transformations corresponding to $\rho_V$ and $\rho'_V := \xi^{G}_V G(\rho_V)$, respectively, via the bijection~\eqref{eq:Hom-V-ZV}. This lemma is equivalent to that the equation
  \begin{equation}
    \label{eq:G-tilde-Z-comod-2}
    G_2(X, V) \circ \sigma'_V(G(X)) = G(\sigma_V(X)) \circ G_2(V, X)
  \end{equation}
  holds for all $X \in \mathcal{C}$. We compute:
  \begin{align*}
    & G_2(X, V) \circ \sigma_V'(G(X)) \circ G_2(V, X)^{-1} \\
    & = (G_2(X, V) \otimes \eval_{G(X)})
      \circ ((\pi_{G(V)}(G(X)) \circ \rho') \otimes \id_{G(X)}) \circ G_2(V, X)^{-1} \\
    & = (G_2(X, V) \otimes \eval_{G(X)})
      \circ (\id_{G(X)} \otimes \id_{G(V)} \otimes \gamma_X \otimes \id_{G(X)}) \\
    & \qquad \circ (G_3(X, V, X^*)^{-1} \otimes \id_{G(X)})
      \circ (G(\pi_V(X) \circ \rho) \otimes \id_{G(X)}) \circ G_2(V, X)^{-1} \\
    & = (G_2(X, V) \otimes G_0^{-1}G(\eval_X))
      \circ (\id_{G(X)} \otimes \id_{G(V)} \otimes G_2(X^*, X)) \\
    & \qquad \circ G_4(X, V, X^*, V)^{-1}
    \circ G((\pi_V(X) \circ \rho)) \otimes \id_{G(X)}) \\
    & = G((\id_{X} \otimes \id_V \otimes \eval_X) \circ ((\pi_V(X) \circ \rho) \otimes \id_X)) = G(\sigma_V(X)).
  \end{align*}
  Here, the first equality follows from the definition of $\sigma'_V$, the second from the definition of $G_3'$, the third from the definition of $\gamma^G$, the fourth from the definition of a monoidal functor, and the last from the definition of $\sigma_V$. Hence~\eqref{eq:G-tilde-Z-comod-2} follows. The proof is done.
\end{proof}

\subsection{Well-definedness of $\Phi$}
\label{subsec:well-def-Phi}

We show that the map~\eqref{eq:Phi} is well-defined, that is, the natural isomorphism $\Phi(\boldsymbol{\beta})$ of Theorem~\ref{thm:main-thm} is a monoidal natural transformation depending only on the equivalence class of $\boldsymbol{\beta} \in \mathscr{J}(F, G)$.

Let $\beta$ be an invertible object of $\mathcal{C}$. We consider the braided tensor autoequivalence of $\mathcal{Z}(\mathcal{C})$ induced by $I^{\beta}$. For an object $\mathbf{V} = (V, \sigma_V) \in \mathcal{Z}(\mathcal{C})$, we define
\begin{equation}
  \label{eq:mon-nat-iso-chi}
  \chi^{\beta}_{\mathbf{V}} = (\sigma_V(\beta) \otimes \id_{\beta^*}) \circ (\id_V \otimes \coev_{\beta}).
\end{equation}

\begin{lemma}
  \label{lem:chi-beta-monoidal}
  $\chi^{\beta}$ is a monoidal natural transformation from $\id_{\mathcal{Z}(\mathcal{C})}$ to $\widetilde{I^{\beta}}$.
\end{lemma}
\begin{proof}
  Let $\mathbf{V} = (V, \sigma_V) \in \mathcal{Z}(\mathcal{C})$ be an object. For all $X \in \mathcal{C}$, we have
  \begin{equation*}
    \newcommand{\xylabel}[1]{\scriptstyle\mathstrut#1}
    \knotholesize{6pt}
    \begin{xy} /r1em/:
      (0,2.5)="P1" *+!D{\xylabel{V}};
      p+(3,0)="P2" *+!D{\xylabel{\beta}};
      p+(1,0)="P3" *+!D{\xylabel{X}};
      p+(1,0)="P4" *+!D{\xylabel{\,\beta^*}},
      "P1"; p-(0,1) **\dir{-}
      ?> \pushrect{1}{1},
      \vtwist~{s0}{s1}{s2}{s3},
      s1; p+(1,0) \mycap{+.75}
      ?>; p-(0,1) **\dir{-}
      ?>; p+(1,0) \mycap{-.75}
      ?>; "P2" **\dir{-},
      s2; p-(0,3) **\dir{-}
      ?> *+!U{\xylabel{\beta}},
      s3 \pushrect{3}{3},
      \vtwist~{s0}{s1}{s2}{s3},
      s2 *+!U{\xylabel{X}},
      s3 *+!U{\xylabel{V}},
      s1; "P3" **\dir{-},
      s3+(1,0); "P4" **\dir{-}
      ?< *+!U{\xylabel{\,\beta^*}},
    \end{xy}
    =
    \begin{xy} /r1em/:
      (0,2.5)="P1" *+!D{\xylabel{V}};
      p+(1,0)="P2" *+!D{\xylabel{\beta}};
      p+(1,0)="P3" *+!D{\xylabel{X}};
      p+(1,0)="P4" *+!D{\xylabel{\,\beta^*}},
      "P1" \pushrect{1}{1},
      \vtwist~{s0}{s1}{s2}{s3},
      s2; p-(0,4) **\dir{-} ?>="Q1"
      *+!U{\xylabel{\beta}},
      s3 \pushrect{1}{1},
      \vtwist~{s0}{s1}{s2}{s3},
      s1; "P3" **\dir{-},
      s2; \makecoord{p}{"Q1"} **\dir{-}
      ?> *+!U{\xylabel{X}},
      s3 \pushrect{1}{1},
      \vtwist~{s0}{s1}{s2}{s3},
      s1; "P4" **\dir{-},
      s3 \pushrect{1}{1},
      \vtwist~{s0}{s1}{s2}{s3},
      s6; \makecoord{p}{s2} **\dir{-}
      ?>; s2 \mycap{-.75},
      s3; \makecoord{p}{"Q1"} **\dir{-}
      ?> *+!U{\xylabel{V}},
      s1; p+(1,0) \mycap{+.75}
      ?>; \makecoord{p}{"Q1"} **\dir{-}
      ?> *+!U{\xylabel{\,\beta^*}}
    \end{xy}
  \end{equation*}
  and hence the diagram
  \begin{equation*}
    \knotholesize{.75em}
    \xymatrix{
      V \otimes X^{\beta}
      \ar[d]_{\sigma_{V}(X^{\beta})}
      \ar[rr]^{\chi_{\mathbf{V}}^{\beta} \otimes \id}
      & & V^{\beta} \otimes X^{\beta}
      \ar[rr]^{I^{\beta}_2(V,X)}
      & & (V \otimes X)^{\beta}
      \ar[d]^{\id \otimes \sigma_{V}(X) \otimes \id} \\
      X^{\beta} \otimes V
      \ar[rr]_{\id \otimes \chi_{\mathbf{V}}^{\beta}}
      & & X^{\beta} \otimes V^{\beta}
      \ar[rr]_{I^{\beta}_2(X,V)}
      & & (X \otimes V)^{\beta}
    }
  \end{equation*}
  commutes. Thus, by~\eqref{eq:def-half-br-F-tilde}, we have $\chi^{\beta}_{\mathbf{V}} \in \Hom_{\mathcal{Z}(\mathcal{C})}(\mathbf{V}, \widetilde{I^{\beta}}(\mathbf{V}))$.

  By the definition of morphisms in $\mathcal{Z}(\mathcal{C})$, it is easy to see that the morphism $\chi^{\beta}_{\mathbf{V}}$ is natural in $\mathbf{V} \in \mathcal{Z}(\mathcal{C})$. To complete the proof, we show that $\chi^{\beta}$ is monoidal. It is easy to see that $\chi_{(\unitobj, \id)}^{\beta} = I^{\beta}_0$. We also have
  \begin{equation*}
    \newcommand{\xylabel}[1]{\scriptstyle\mathstrut#1}
    I_2^{\beta}(V, W)
    \circ (\chi_{\mathbf{V}}^{\beta} \otimes \chi_{\mathbf{W}}^{\beta})
    = \knotholesize{6pt}
    \begin{xy} /r1em/:
      (0,1.5)="P1" *+!D{\xylabel{V}};
      p+(3,0)="P2" *+!D{\xylabel{W}},
      "P1"; p-(0,1) **\dir{-}
      ?> \pushrect{1}{1},
      \vtwist~{s0}{s1}{s2}{s3},
      s1; p+(1,0) \mycap{+.75}
      ?>; p-(0,1) **\dir{-}
      ?>; p+(1,0) \mycap{-.75},
      s2; p-(0,1) **\dir{-}
      ?>="Q1" *+!U{\xylabel{\beta}},
      s3; \makecoord{p}{"Q1"} **\dir{-}
      ?> *+!U{\xylabel{V}},
      "P2"; p-(0,1) **\dir{-}
      ?> \pushrect{1}{1},
      \vtwist~{s0}{s1}{s2}{s3},
      s1; p+(1,0) \mycap{+.75}
      ?>; \makecoord{p}{"Q1"} **\dir{-}
      ?> *+!U{\xylabel{\beta^*}},
      s3; \makecoord{p}{"Q1"} **\dir{-}
      ?> *+!U{\xylabel{W}},
    \end{xy}
    = \begin{xy} /r1em/:
      (0,1.5)="P1" *+!D{\xylabel{V}};
      p+(1,0)="P2" *+!D{\xylabel{W}},
      "P2"; p-(0,.5) **\dir{-}
      ?> \pushrect{1}{1},
      \vtwist~{s0}{s1}{s2}{s3},
      s3; p-(0,1.5) **\dir{-}
      ?>  ="Q3" *+!U{\xylabel{W}},
      s1; p+(1,0) \mycap{+.75}
      ?>; \makecoord{p}{"Q3"} **\dir{-}
      ?>  *+!U{\xylabel{\beta^*}},
      "P1"; \makecoord{p}{s2} **\dir{-}
      ?> \pushrect{1}{1},
      \vtwist~{s0}{s1}{s2}{s3},
      s2; p-(0,.5) **\dir{-}
      ?> *+!U{\xylabel{\beta}},
      s3; p-(0,.5) **\dir{-}
      ?> *+!U{\xylabel{V}},
    \end{xy}
    = \chi^{\beta}_{\mathbf{V} \otimes \mathbf{W}}
  \end{equation*}
  for all objects $\mathbf{V}, \mathbf{W} \in \mathcal{Z}(\mathcal{C})$. Therefore $\chi^{\beta} \in \Nat_{\otimes}(\id_{\mathcal{Z}(\mathcal{C})}, \widetilde{I^{\beta}})$.
\end{proof}

The category $\cAut_{\otimes}(\mathcal{C})$ is a strict monoidal category with the tensor product given by the composition. The category $\cAut_{\otimes}^{\mathrm{br}}(\mathcal{Z}(\mathcal{C}))$ of braided tensor autoequivalences of $\mathcal{Z}(\mathcal{C})$ is also a strict monoidal category in the same way. The assignment $F \mapsto \widetilde{F}$ is in fact a strict monoidal functor
\begin{equation*}
  \widetilde{\phantom{x}}: \cAut_{\otimes}(\mathcal{C}) \to \cAut_{\otimes}^{\mathrm{br}}(\mathcal{Z}(\mathcal{C})),
  \quad F \mapsto \widetilde{F}
\end{equation*}
in a natural way. Now let $F, G \in \cAut_{\otimes}(\mathcal{C})$ and $\boldsymbol{\beta} = (\beta, j) \in \mathscr{J}(F, G)$. Then the natural transformation $\Phi(\boldsymbol{\beta})$ of Theorem~\ref{thm:main-thm} is written as the composition
\begin{equation}
  \label{eq:Phi-beta-chi}
  \Phi(\boldsymbol{\beta}) = \left(
    \widetilde{F}
    \xrightarrow{\quad \id \quad}
    \id_{\mathcal{Z}(\mathcal{C})} \circ \widetilde{F}
    \xrightarrow{\quad \chi^{\beta} \quad}
    \widetilde{I^{\beta}} \circ \widetilde{F}
    \xrightarrow{\quad \id \quad}
    \widetilde{I^{\beta} F}
    \xrightarrow{\quad \widetilde{j} \quad}
    \widetilde{G}
  \right).
\end{equation}
Hence $\Phi(\boldsymbol{\beta}): \widetilde{F} \to \widetilde{G}$ is a monoidal natural transformation as a composition of monoidal natural transformations.

Finally, we show that $\Phi(\boldsymbol{\beta}) \in \Nat_{\otimes}(\widetilde{F}, \widetilde{G})$ depends only on the equivalence class of $\boldsymbol{\beta} \in \mathscr{J}(F, G)$. This can be proved directly, but the computation will be easier if we notice the following two lemmas: Let $\beta$ be an invertible object of $\mathcal{C}$. If we identify $\mathcal{Z}(\mathcal{C})$ with $\mathcal{C}^Z$, then:

\begin{lemma}
  \label{lem:chi-1-1}
  For a $Z$-comodule $\mathbf{V} = (V, \rho_V) \in \mathcal{C}^Z$, we have
  \begin{equation*}
    \chi^{\beta}_{\mathbf{V}} = \pi_V(\beta) \circ \rho_V.
  \end{equation*}
\end{lemma}
\begin{proof}
  This lemma is graphically proved as follows:
  \begin{equation*}
    \chi^{\beta}_{\mathbf{V}}
    = \xy /r1.5pc/:
    (0,1) \pushrect{1}{1},
    \vtwist~{s0}{s1}{s2}{s3},
    s0; p+(0,1) **\dir{-} ?> *+!D{\scriptstyle V},
    s1; p+(1,0) \mycap{.75} ?>; p-(0,2) **\dir{-}
    ?> *+!U{\scriptstyle \beta^*},
    s2; p-(0,1) **\dir{-} ?> *+!U{\scriptstyle \beta},
    s3; p-(0,1) **\dir{-} ?> *+!U{\scriptstyle V}
    \endxy
    \mathop{=}^{\text{\eqref{eq:def-sigma-from-rho}}}
    \ \xy /r1.5pc/:
    (0,2)="P1" *+!D{\scriptstyle V},
    "P1"-(0,1.5)="P2" *+!U{\scriptstyle \pi_{V}(\beta)} *\frm{-}="BX1",
    "P1"; "P2" **\dir{} ?(.5) *+!{\scriptstyle \rho_V} *\frm{-}="BX2",
    "BX2"!U; "P1" **\dir{-},
    "BX2"!D; "P2" **\dir{-},
    "BX1"!D!L(.6); \makecoord{p}{(0,-1)}
    **\dir{-} ?> *+!U{\scriptstyle \beta},
    "BX1"!D!L(.0); \makecoord{p}{(0,-1)}
    **\dir{-} ?> *+!U{\scriptstyle V},
    "BX1"!D!R(.6); p+(.75,0) \mycap{-.75}
    ?>; \makecoord{p}{(0,1)} **\dir{-}
    ?> *+!R{\scriptstyle \beta}
    ?>; p+(1,0) \mycap{.75}
    ?>; \makecoord{p}{(0,-1)} **\dir{-}
    ?> *+!U{\scriptstyle \beta^*},
    \endxy
    = \pi_V(\beta) \circ \rho_V.
    \qedhere
  \end{equation*}
\end{proof}

\begin{lemma}
  \label{lem:chi-1-2}
  For $\boldsymbol{\beta} = (\beta, j) \in \mathscr{J}(F, G)$ and $\mathbf{V} = (V, \rho_V) \in \mathcal{C}^Z$, we have
  \begin{equation*}
    \Phi(\boldsymbol{\beta})_{\mathbf{V}} = \chi^{\beta}_V \circ \pi_{V}(\beta) \circ \xi_V^F \circ F(\rho_{V}),
  \end{equation*}
  where $\xi^F: F Z \to Z F$ is the natural isomorphism introduced in Subsection~\ref{subsec:comon-desc-G-tilde}.
\end{lemma}
\begin{proof}
  This follows from equation \eqref{eq:Phi-beta-chi} and Lemma~\ref{lem:chi-1-1}.
\end{proof}

Let $\boldsymbol{\beta} = (\beta, j)$ and $\boldsymbol{\beta}' = (\beta', j')$ be elements of $\mathscr{J}(F, G)$ such that $\boldsymbol{\beta} \sim \boldsymbol{\beta}'$. By the definition of $\sim$, there is an isomorphism $c: \beta \to \beta'$ in $\mathcal{C}$ satisfying \eqref{eq:def-J-F-G-equiv}. We prove $\Phi(\boldsymbol{\beta}) = \Phi(\boldsymbol{\beta}')$ as follows: For all objects $\mathbf{V} \in \mathcal{Z}(\mathcal{C})$, we have
\begin{align*}
  \Phi(\boldsymbol{\beta}')_{\mathbf{V}}
  & = j'_{V} \circ \pi_{V}(\beta') \circ \xi^F_V \circ F(\rho_{V})
  \qquad \text{(by Lemma \ref{lem:chi-1-2})} \\
  & = j_X \circ (c^{-1} \otimes \id_{F(X)} \otimes c^*) \circ \pi_{V}(\beta') \circ \xi^F_V \circ F(\rho_{V}) \\
  & = j_X \circ \pi_{V}(\beta) \circ \xi^F_V \circ F(\rho_{V})
  \qquad \phantom{'} \text{(by the dinaturality of $\pi_{V}$)} \\
  & = \Phi(\boldsymbol{\beta})_{\mathbf{V}},
\end{align*}
where $\rho_V$ is the $Z$-coaction associated to the half-braiding of $\mathbf{V}$.

\subsection{Reduction to the case where $F$ is the identity}
\label{subsec:reduction}

We explain that it is sufficient to prove Theorem~\ref{thm:main-thm} in the case where $F = \id_{\mathcal{C}}$. We first note that
\begin{equation*}
  N: \cAut_{\otimes}(\mathcal{C})^{\op} \times \cAut_{\otimes}(\mathcal{C}) \to \Ens,
  \quad (F, G) \mapsto \Nat_{\otimes}(\widetilde{F}, \widetilde{G})
\end{equation*}
is a functor in an obvious way. We also note that
\begin{equation*}
  J: \cAut_{\otimes}(\mathcal{C})^{\op} \times \cAut_{\otimes}(\mathcal{C}) \to \Ens,
  \quad (F, G) \mapsto J(F, G)
\end{equation*}
is a functor. Indeed, if $\xi: F' \to F$ and $\zeta: G \to G'$ are morphisms in $\cAut_{\otimes}(\mathcal{C})$, then there is the map
\begin{equation*}
  J(\xi, \zeta): J(F, G) \to J(F', G'),
  \quad [(\beta, j)] \mapsto [(\beta, \zeta_{(-)}j_{(-)}I^{\beta}(\xi_{(-)}))]
\end{equation*}
The equation $J(\xi, \zeta) \circ J(\xi', \zeta') = J(\xi' \xi, \zeta \zeta')$ holds for all morphisms $\xi$, $\xi'$, $\zeta$ and $\zeta'$ in $\cAut_{\otimes}(\mathcal{C})$ whenever $\xi' \xi$ and $\zeta \zeta'$ are defined.

\begin{lemma}
  \label{lem:cat-prop-Phi-1}
  The map $[\boldsymbol{\beta}] \mapsto \Phi(\boldsymbol{\beta})$ is a natural transformation from $J$ to $N$.
\end{lemma}
\begin{proof}
  For an element $\boldsymbol{\beta} = (\beta, j) \in \mathscr{J}(F, G)$, an object $\mathbf{V} = (V, \sigma_V) \in \mathcal{Z}(\mathcal{C})$ and morphisms $\xi: F' \to F$ and $\zeta: G \to G'$ in $\cAut_{\otimes}(\mathcal{C})$, we have
  \begin{align*}
    \Phi(J(\xi, \zeta)(\boldsymbol{\beta}))_{\mathbf{V}}
    = {\knotholesize{10pt} \xy /r1pc/: (0,3)="T1"; p-(0,6)="T0",
    "T1" *+!D{\scriptstyle F'(V)}; p-(0, .25) **\dir{-}
    ?> \pushrect{1.5}{1.5}, \vtwist~{s0}{s1}{s2}{s3},
    s3 *+!U{\scriptstyle \xi_V} *\frm{-}="BX1",
    "BX1"!D; \makecoord{p}{s2-(0,1.5)} **\dir{-}
    ?> *+!U{\makebox[2.5pc]{$\scriptstyle j_V$}} *\frm{-}="BX2",
    "BX2"!D; \makecoord{p}{"T0"} **\dir{}
    ?(.55) *+!{\scriptstyle \zeta_V} *\frm{-}="BX3",
    "BX3"!U; "BX2"!D **\dir{-},
    "BX3"!D; \makecoord{p}{"T0"} **\dir{-}
    ?> *+!U{\scriptstyle G'(V)},
    "BX2"!U!L(.6)="P1",
    "BX2"!U!R(.6)="P2",
    s2; "P1" \mybend{.5},
    s1; p+(1,0) \mycap{.75}
    ?>; "P2" \mybend{.5},
    s2 *+!R{\scriptstyle \beta}
    \endxy}
    = {\knotholesize{8pt} \xy /r1pc/: (0,3)="T1"; p-(0,6)="T0",
    "T1" *+!D{\scriptstyle F'(V)};
    p-(0,.5) **\dir{-}
    ?> *+!U{\scriptstyle \xi_V} *\frm{-}="BX1",
    "BX1"!D \pushrect{1}{1}, \vtwist~{s0}{s1}{s2}{s3},
    s3; p-(0,.725) **\dir{-}
    ?> *+!U{\makebox[2.5pc]{$\scriptstyle j_V$}} *\frm{-}="BX2",
    "BX2"!D; \makecoord{p}{"T0"} **\dir{}
    ?(.55) *+!{\scriptstyle \zeta_V} *\frm{-}="BX3",
    "BX3"!U; "BX2"!D **\dir{-},
    "BX3"!D; \makecoord{p}{"T0"} **\dir{-}
    ?> *+!U{\scriptstyle G'(V)},
    "BX2"!U!L(.6)="P1",
    "BX2"!U!R(.6)="P2",
    s2; "P1" \mybend{.5},
    s1; p+(1,0) \mycap{.75}
    ?>; "P2" \mybend{.5}
    \endxy}
    = \widetilde{\zeta}_{\mathbf{V}} \circ \Phi(\boldsymbol{\beta})_{\mathbf{V}} \circ \widetilde{\xi}_{\mathbf{V}},
  \end{align*}
  where the second equality follows since $\xi_V \in \Hom_{\mathcal{Z}(\mathcal{C})}(\widetilde{F'}(\mathbf{V}), \widetilde{F}(\mathbf{V}))$. The proof is done.
\end{proof}

Now let $F$, $G$ and $H$ be objects of $\cAut_{\otimes}(\mathcal{C})$. There are maps
\begin{align}
  \label{eq:induced-map-H-1}
  J(F, G) & \to J(F H, G H),
  & [(\beta, j)] & \mapsto [(\beta, j_{H(-)})], \\
  \label{eq:induced-map-H-2}
  N(F, G) & \to N(F H, G H),
  & \xi & \mapsto \xi_{\widetilde{H}(-)}.
\end{align}
Since $H$ is a tensor autoequivalence, both these maps are bijective. Moreover, these maps are compatible with $\Phi$ in the following sense:

\begin{lemma}
  \label{lem:cat-prop-Phi-2}
  Let $F$, $G$ and $H$ be as above. The following diagram commutes\textup{:}
  \begin{equation*}
    \xymatrix@C=64pt{
      J(F, G)
      \ar[d]_{\Phi_{F, G}}
      \ar[r]^{\text{\eqref{eq:induced-map-H-1}}}
      & J(F H, G H)
      \ar[d]^{\Phi_{F H, G H}} \\
      N(F, G)
      \ar[r]^{\text{\eqref{eq:induced-map-H-2}}}
      & N(F H, G H)
    }
  \end{equation*}
\end{lemma}
\begin{proof}
  The proof is immediate from the definition of $\Phi$.
\end{proof}

Let $F$ and $G$ be as above, and let $\overline{F}$ be a quasi-inverse of $F$. Then $\overline{F}$ is also a tensor autoequivalence of $\mathcal{C}$ such that $\overline{F} F \cong \id_{\mathcal{C}} \cong F \overline{F}$ as monoidal functors. We choose an isomorphism $\eta: \id_{\mathcal{C}} \to F \overline{F}$ of monoidal functors. Then, by Lemmas \ref{lem:cat-prop-Phi-1} and \ref{lem:cat-prop-Phi-2}, we have the following commutative diagram:
\begin{equation*}
  \xymatrix@C=64pt{
    J(F, G)
    \ar[d]_{\Phi_{F, G}}
    \ar[r]^(.45){\text{\eqref{eq:induced-map-H-1} with $H = \overline{F}$}}
    & J(F \overline{F}, G \overline{F})
    \ar[d]^{\Phi_{F \overline{F}, G \overline{F}}}
    \ar[r]^{J(\eta, \id)}
    & J(\id_{\mathcal{C}}, G \overline{F})
    \ar[d]^{\Phi_{\id_{\mathcal{C}}, G \overline{F}}} \\
    N(F, G)
    \ar[r]^(.45){\text{\eqref{eq:induced-map-H-2} with $H = \overline{F}$}}
    & N(F \overline{F}, G \overline{F})
    \ar[r]^{N(\eta, \id)}
    & N(\id_{\mathcal{C}}, G \overline{F}) \\
  }
\end{equation*}
Since horizontal arrows in this diagram are bijections, we see that $\Phi_{F, G}$ is bijective if and only if $\Phi_{\id_{\mathcal{C}}, G \overline{F}}$ is bijective. Thus, to prove Theorem~\ref{thm:main-thm}, it is enough to consider the case where $F = \id_{\mathcal{C}}$.

\subsection{Technical lemmas}
\label{subsec:technical-lemmas}

We now prepare several technical lemmas to prove Theorem~\ref{thm:main-thm}. We fix a tensor autoequivalence $G$ of $\mathcal{C}$. We have introduced a natural isomorphism $\xi^G: G Z \to Z G$ in Subsection \ref{subsec:comon-desc-G-tilde}. 

\begin{lemma}
  \label{lem:xi-G-monoidal}
  $\xi^G$ is an isomorphism of monoidal functors from $\widetilde{G} R$ to $R G$.
\end{lemma}
\begin{proof}
  For any functors $E: \mathcal{C} \to \mathcal{C}$ and $F: \mathcal{Z}(\mathcal{C}) \to \mathcal{Z}(\mathcal{C})$, the map
  \begin{equation*}
    \Nat(U F, E U)
    \to \Nat(F R, R E);
    \quad \theta \mapsto R E(\varepsilon_{(-)}^{}) R(\theta_{R(-)}^{}) \eta_{F R(-)}^{}
  \end{equation*}
  is bijective. Since $U \dashv R$ is a monoidal adjunction, this restricts to the bijection
  \begin{equation}
    \label{eq:xi-G-monoidal-1}
    \Nat_{\otimes}(U F, E U) \to \Nat_{\otimes}(F R, R E)
  \end{equation}
  when both $E$ and $F$ are monoidal functors.

  Now we consider the case where $E = G$ and $F = \widetilde{G}$. Then, since $U \widetilde{G} = G U$ as monoidal functors, the identity natural transformation $\id_{G U}$ lives in the source of the map~\eqref{eq:xi-G-monoidal-1}. Let $\xi'$ be the monoidal natural transformation corresponding to $\id_{G U}$ via \eqref{eq:xi-G-monoidal-1}. By Lemma~\ref{lem:chi-1-2} and the definition of $R$, we have
  \begin{equation*}
    \widetilde{G}R(V) = (G Z(V), \xi^G_{Z(V)} G(\delta_{V}))
  \end{equation*}
  for all $V \in \mathcal{C}$. Thus $\xi'_{V}$ for $V \in \mathcal{C}$ is computed as follows:
  \begin{align*}
    \xi'_V
    = Z G(\varepsilon_V) \circ \xi^G_{Z(V)} \circ G(\delta_{V})
    = \xi^G_V \circ G Z(\varepsilon_V) \circ G(\delta_{V})
    = \xi^G_V \circ G(\id_{Z(V)}) = \xi^G_V.
  \end{align*}
  Here, the first equality follows from the definition of $\xi'$, the second from the naturality of $\xi$, and the third from the axiom of a comonad. Since $\xi'$ belongs to the set $\Nat_{\otimes}(\widetilde{G} R, R G)$, so does $\xi^G$.
\end{proof}

Let $\beta \in \mathcal{C}$ be an invertible object. We remark:

\begin{lemma}
  \label{lem:chi-RV}
  For all objects $V, X \in \mathcal{C}$, we have
  \begin{equation*}
    (\id_{\beta} \otimes \pi_V(X) \otimes \id_{\beta^*}) \circ \chi_{R(V)}^{\beta} = \pi_{V}(\beta \otimes X).
  \end{equation*}
\end{lemma}
\begin{proof}
  Since $R(V) = (Z(V), \delta_V)$, we compute
  \begin{align*}
    (\id_{\beta} \otimes \pi_V(X) \otimes \id_{\beta^*}) \circ \chi_{R(V)}^{\beta}
    & = (\id_{\beta} \otimes \pi_V(X) \otimes \id_{\beta^*}) \circ \pi_{Z(V)}(\beta) \circ \delta_V \\
    & = \pi_{V}(\beta \otimes X)
  \end{align*}
  for all $V, X \in \mathcal{C}$. Here, the first and the second equalities follow from Lemma~\ref{lem:chi-1-1} and equation~\eqref{eq:Hopf-comonad-Z-delta}, respectively.
\end{proof}

The following identity will be useful:

\begin{lemma}
  \label{lem:chi-2}
  For $V, X \in \mathcal{C}$ and $\boldsymbol{\beta} = (\beta, j) \in \mathscr{J}(\id_{\mathcal{C}}, G)$, we have
  \begin{equation}
    \label{eq:lem-chi-2}
    G(\pi_V(X)) \circ \Phi(\boldsymbol{\beta})_{R(V)} = j_{X \otimes V \otimes X^*} \circ \pi_{V}(\beta \otimes X).
  \end{equation}
\end{lemma}
\begin{proof}
  By Lemmas \ref{lem:chi-1-2} and \ref{lem:chi-RV}, we have
  \begin{align*}
    G(\pi_V(X)) \circ \Phi(\boldsymbol{\beta})_{R(V)}
    & = G(\pi_V(X)) \circ j_V \circ \pi_{Z(V)}(\beta) \circ \delta_V \\
    & = j_{X \otimes V \otimes X^*} \circ (\id_{\beta} \otimes \pi_V(X) \otimes \id_{\beta^*})
      \circ \pi_{Z(V)}(\beta) \circ \delta_V \\
    & = j_{X \otimes V \otimes X^*} \circ \pi_{V}(\beta \otimes X). \qedhere
  \end{align*}
\end{proof}

Let $\beta \in \mathcal{C}$ be an invertible object. We note that $(-) \otimes \beta: \mathcal{C} \to \mathcal{C}$ is an autoequivalence of $\mathcal{C}$. Thus, by Lemma~\ref{lem:end-equiv-twist}, we can define a morphism $\zeta^{\beta}_V: Z I^{\beta}(V) \to Z(V)$ in $\mathcal{C}$ for $V \in \mathcal{C}$ to be the unique morphism in $\mathcal{C}$ such that the diagram
\begin{equation*}
  \xymatrix@C=64pt{
    Z(V^{\beta})
    \ar[r]^(.35){\pi_{V^{\beta}}(X)} \ar[d]_{\zeta^{\beta}_V}
    & X \otimes V^{\beta} \otimes X^* \ar@{=}[d] \\
    Z(V)
    \ar[r]_(.35){\pi_{V}(X \otimes \beta)}
    & X \otimes \beta \otimes V \otimes (X \otimes \beta)^* \\
  }
\end{equation*}
commutes for all $X \in \mathcal{C}$. It is routine to check that $\zeta^{\beta}_V$ is an isomorphism in $\mathcal{C}$ and natural in the variable $V$.

\begin{lemma}
  \label{lem:zeta-monoidal}
  $\zeta^{\beta} \in \Nat_{\otimes}(R I^{\beta}, R)$.
\end{lemma}
\begin{proof}
  By abuse of notation, we write $\xi^{\beta} = \xi^{F}$ if $F = I^{\beta}$. Then the diagram
  \begin{equation*}
    \xymatrix@R=24pt@C=56pt{
      Z(V) \ar[r]^(.25){\pi_{V}(\beta \otimes X)}
      \ar@{}@<.5ex>[rd]|{\text{($\circlearrowleft$ by Lemma~\ref{lem:chi-RV})}}
      \ar[d]_{\chi^{\beta}_{R(V)}}
      & \beta \otimes X \otimes V \otimes (\beta \otimes X)^*
      \ar@{=}[d] \\
      Z(V)^{\beta} \ar[r]^(.25){(\pi_V(X))^{\beta}}
      \ar[d]_{\xi_{V}^{\beta}}
      \ar@{}@<.5ex>[rd]|{\text{($\circlearrowleft$ by the definition of $\xi^{\beta}$)}}
      & \beta \otimes X \otimes V \otimes X^* \otimes \beta
      \ar[d]^{\id_{\beta} \otimes \id_X \otimes \eval_{\beta}^{-1} \otimes \id_V \otimes \coev_{\beta^*} \otimes \id_{X^*} \otimes \id_{\beta^*}} \\
      Z(V^{\beta}) \ar[r]^(.25){\pi_{V^{\beta}}(X^{\beta})}
      & \beta \otimes X \otimes \beta^* \otimes \beta \otimes V \otimes \beta^* \otimes \beta^{**} \otimes X^* \otimes \beta^*
    }
  \end{equation*}
  commutes for all $V, X \in \mathcal{C}$ (see Lemma~\ref{lem:conj-beta-duality-trans} for the duality transformation of $I^{\beta}$). Thus, by the dinaturality of $\pi$ and $\coev_{\beta^*} = (\eval_{\beta})^*$, we have
  \begin{equation*}
    \pi_{V^{\beta}}(X^{\beta}) \circ \xi^{\beta}_V \circ \chi^{\beta}_{R(V)} = \pi_{V}(\beta \otimes X \otimes \beta^* \otimes \beta) = \pi_{V}(X^{\beta} \otimes \beta)
  \end{equation*}
  for all $V, X \in \mathcal{C}$. Since $I^{\beta}$ is an equivalence, we may replace $X^{\beta}$ in the above equation with $X$ and obtain the following equation:
  \begin{equation*}
    \pi_{V^{\beta}}(X) \circ \xi^{\beta}_V \circ \chi^{\beta}_{R(V)} = \pi_{V}(X \otimes \beta).
  \end{equation*}
  By comparing this equation with the definition of $\zeta^{\beta}$, we have
  \begin{equation*}
    (\zeta^{\beta})^{-1} = \xi^{\beta}_{(-)} \circ \chi^{\beta}_{R(-)} \in \Nat_{\otimes}(R, R I^{\beta})
  \end{equation*}
  and hence $\zeta^{\beta} \in \Nat_{\otimes}(R I^{\beta}, R)$.
\end{proof}

\subsection{Construction of the inverse}
\label{subsec:Phi-inv}

As we have discussed in Subsection~\ref{subsec:reduction}, we may assume $F = \id_{\mathcal{C}}$ to prove Theorem~\ref{thm:main-thm}. Thus, till the end of this section, we fix a tensor autoequivalence $G$ of $\mathcal{C}$. We introduce the following notation:

\begin{notation*}
  We fix a complete set $\{ \beta_1, \dotsc, \beta_n \}$ of representatives of isomorphism classes of invertible objects of $\mathcal{C}$. We define the sets $J_G$ and $N_G$ by
  \begin{equation*}
    J_G := \bigsqcup_{s = 1}^n \{ (\beta_s, j) \mid j \in \Nat_{\otimes}(I^{\beta_s}, G) \}
    \quad \text{and} \quad
    N_G := \Nat_{\otimes}(\id_{\mathcal{Z}(\mathcal{C})}, \widetilde{G}),
  \end{equation*}
  respectively, and define the map $\Phi_G: J_G \to N_G$ by $\boldsymbol{\beta} \mapsto \Phi(\boldsymbol{\beta})$.
\end{notation*}

The set $J(\id_{\mathcal{C}}, G)$ is identified with $J_G$. To prove Theorem~\ref{thm:main-thm}, it is sufficient to show that the map $\Phi_G$ is bijective. Now we construct the inverse of $\Phi_G$. Given $h \in N_G$, we associate an invertible object $\beta \in \{ \beta_1, \dotsc, \beta_n \}$ as follows: We consider the composition
\begin{equation}
  \label{eq:def-Psi-beta-0}
  \mathbf{A} = R(\unitobj)
  \xrightarrow{\quad h_{R(\unitobj)^{}} \quad}
  \widetilde{G} R(\unitobj)
  \xrightarrow{\quad \xi^{G}_{\unitobj} \quad}
  R G(\unitobj)
  \xrightarrow{\quad R(G_0^{-1}) \quad}
  R(\unitobj)
  = \mathbf{A}
\end{equation}
of morphisms in $\mathcal{Z}(\mathcal{C})$. Since $h$ is a monoidal natural transformation, and since so is $\xi^{G}$ by Lemma~\ref{lem:xi-G-monoidal}, the morphism \eqref{eq:def-Psi-beta-0} is a morphism of algebras in $\mathcal{Z}(\mathcal{C})$. Thus, by Lemma~\ref{lem:Alg-ZC-A-A}, there exists a unique invertible object $\beta \in \{ \beta_1, \dotsc, \beta_n \}$ such that the following two equivalent equations hold:
\begin{gather}
  \label{eq:def-Psi-beta-1}
  Z(G_0^{-1}) \circ \xi_{\unitobj}^F \circ h_{R(\unitobj)}
  = Z(\intch_{\beta}) \circ \delta_{\unitobj}, \\
  \label{eq:def-Psi-beta-2}
  \varepsilon_{\unitobj} \circ Z(G_0^{-1}) \circ \xi_{\unitobj}^F \circ h_{R(\unitobj)}
  = \intch_{\beta}
\end{gather}
By using $\beta$ defined in the above, we define $\phi: R I^{\beta} \to R G$ by
\begin{equation*}
  \phi_V: R(V^{\beta})
  \xrightarrow{\quad \zeta^{\beta}_V \quad}
  R(V)
  \xrightarrow{\quad h_{R(V)} \quad}
  \widetilde{F}R(V)
  \xrightarrow{\quad \xi^G_V \quad}
  R G(V)
\end{equation*}
for $V \in \mathcal{C}$. We recall that an object of the form $R(W)$ for some $W \in \mathcal{C}$ is a right $\mathbf{A}$-module ({\it i.e.}, a right Hopf module over $Z$) by the action $R_2(\unitobj, W)$.

\begin{claim}
  $\phi_V$ is an isomorphism of Hopf modules.
\end{claim}
\begin{proof}
  We set $p_V = \xi^{G}_V \circ h_{R(V)}$ and $q_V = \zeta^{\beta}_V$ for $V \in \mathcal{C}$. By Lemmas~\ref{lem:xi-G-monoidal} and~\ref{lem:zeta-monoidal}, the natural transformations $p: R \to R G$ and $q: R I^{\beta} \to R$ are monoidal. Thus the following diagram commutes:
  \begin{equation*}
    \xymatrix{
      R(V^{\beta}) \otimes \mathbf{A}
      \ar[rr]^(.475){\id \otimes R(\coev_{\beta})}
      \ar[d]_{q_V \otimes q'}
      \ar@{}[rrd]|{(q' := q_{\unitobj} \circ R(\coev_{\beta}))}
      & &
      R(V^{\beta}) \otimes R(\unitobj^{\beta})
      \ar[rr]^(.6){(R I^{\beta})_2}
      \ar[d]^{q_V \otimes q_{\unitobj}}
      & &
      R(V^{\beta}) \ar[d]^{q_V} \\
      R(V) \otimes \mathbf{A}
      \ar@{=}[rr]
      \ar[d]_{p_V \otimes p'}
      \ar@{}[rrd]|{(p' := R(G_0^{-1}) \circ p_{\unitobj})}
      & &
      R(V) \otimes R(\unitobj)
      \ar[rr]^(.6){R_2}
      \ar[d]^{p_V \otimes p_{\unitobj}}
      & &
      R(V) \ar[d]^{p_V} \\
      R G(V) \otimes \mathbf{A}
      \ar[rr]_(.475){\id \otimes R(G_0)}
      & &
      R G(V) \otimes R G(\unitobj)
      \ar[rr]_(.6){(R G)_2}
      & & R G(V)
    }
  \end{equation*}
  By the definition of monoidal functors, the composite morphisms along the top row and the bottom row are $R_2(V^{\beta}, \unitobj)$ and $R_2(G(V), \unitobj)$, respectively. Thus we get the following commutative diagram:
  \begin{equation*}
    \xymatrix{
      R(V^{\beta}) \otimes \mathbf{A}
      \ar@{=}[r]
      \ar[d]_{\phi_V \otimes p' q'}
      & R(V^{\beta}) \otimes R(\unitobj)
      \ar[rr]^(.575){R_2(V^{\beta}, \unitobj)}
      & & R(V^{\beta})
      \ar[d]^{\phi_V} \\
      R G(V) \otimes \mathbf{A}
      \ar@{=}[r]
      & R G(V) \otimes R(\unitobj)
      \ar[rr]_(.575){R_2(G(V), \unitobj)}
      & & R G(V)
    }
  \end{equation*}
  To show that $\phi_V: R(V^{\beta}) \to R G(V)$ is an isomorphism of right Hopf modules, it suffices to show that $p' q'$ is the identity. By the definition of $\beta$, we have
  \begin{equation*}
    p' = R(G_0^{-1}) \circ \xi^G_{\unitobj} \circ h_{R(\unitobj)}
    \mathop{=}^{\text{\eqref{eq:def-Psi-beta-1}}}
    Z(\intch_{\beta}) \circ \delta_{\unitobj}
    \mathop{=}^{\text{\eqref{eq:def-character}}}
    Z(\coev_{\beta}^{-1}) \circ Z(\pi_{\unitobj}(\beta)) \circ \delta_{\unitobj}.
  \end{equation*}
  Hence we compute
  \begin{align*}
    \pi_{\unitobj}(X) \circ p' \circ q'
    & = \pi_{\unitobj}(X)
      \circ Z(\coev_{\beta}^{-1}) \circ Z(\pi_{\unitobj}(\beta)) \circ \delta_{\unitobj}
      \circ \zeta^{\beta}_{\unitobj} \circ Z(\coev_{\beta}) \\
    & = (\id_{X} \otimes \coev_{\beta}^{-1} \otimes \id_{X^*}) \\
    & \qquad \circ (\id_{X} \otimes \pi_{\unitobj}(\beta) \otimes \id_{X^*})
      \circ \pi_{Z(\unitobj)}(X) \circ \delta_{\unitobj}
      \circ \zeta^{\beta}_{\unitobj} \circ Z(\coev_{\beta}) \\
    & = (\id_{X} \otimes \coev_{\beta}^{-1} \otimes \id_{X^*})
      \circ \pi_{\unitobj}(X \otimes \beta) \circ \zeta^{\beta}_{\unitobj} \circ Z(\coev_{\beta}) \\
    & = (\id_{X} \otimes \coev_{\beta}^{-1} \otimes \id_{X^*})
      \circ \pi_{\unitobj^{\beta}}(X) \circ Z(\coev_{\beta}) \\
    & = \pi_{\unitobj}(X) \circ Z(\coev_{\beta}^{-1}) \circ Z(\coev_{\beta})
      = \pi_{\unitobj}(X)
  \end{align*}
  for $X \in \mathcal{C}$. Here, the second and the fifth equalities follow from the naturality of $\pi_{(-)}(X)$, the third from \eqref{eq:Hopf-comonad-Z-delta}, and the fourth from the definition of $\zeta^{\beta}$. By the universal property of $\pi_{\unitobj}$, we obtain $p' q' = \id_{Z(\unitobj)}$. The proof is done.
\end{proof}

By the fundamental theorem for Hopf modules (Lemma~\ref{lem:FTHM}), there is a unique morphism $j_V: V^{\beta} \to G(V)$ such that
\begin{equation}
  \label{eq:def-Psi-j}
  Z(j_V) = \phi_V = \xi^{G}_V \circ h_{R(V)} \circ \zeta^{\beta}_V.
\end{equation}
The family $j = \{ j_V \}_{V \in \mathcal{C}}$ is in fact a natural isomorphism $j: I^{\beta} \to G$, since $\phi_V$ is invertible for all $V \in \mathcal{C}$ and natural in $V \in \mathcal{C}$. Moreover, we have:

\begin{claim}
  $j \in \Nat_{\otimes}(I^{\beta}, G)$.
\end{claim}
\begin{proof}
  Let $H: \mathcal{C} \to \mathcal{Z}(\mathcal{C})_{\mathbf{A}}$ be the monoidal equivalence given by Lemma~\ref{lem:FTHM}.  We set $F = I^{\beta}$ for simplicity. As we have seen, $\phi: R F \to R G$ is a monoidal natural transformation such that $\phi_V$ is an isomorphism of Hopf modules for all $V \in \mathcal{C}$. From this, we see that $\phi: H G \to H F$ is in fact a monoidal natural transformation. Hence $j: F \to G$ is also a monoidal natural transformation.
\end{proof}

We now define the map $\Psi_G: N_G \to J_G$ by $\Psi_G(h) = (\beta, j)$, where $\beta$ and $j$ are constructed from $h$ in the above. In the rest of this section, we prove that both $\Phi_G \Psi_G$ and $\Psi_G \Phi_G$ are identity maps.

\subsection{Computation of $\Psi_G \Phi_G$}
\label{subsec:compute-Psi-Phi}

We prove that $\Psi_G \Phi_G$ is the identity map. Let $\boldsymbol{\beta} = (\beta, j)$ be an element of $J_G$, and set
\begin{equation*}
  h = \Phi_G(\boldsymbol{\beta})
  \quad \text{and} \quad
  (\beta', j') = \Psi_G(h).
\end{equation*}
We first show $\beta = \beta'$. We recall that $\beta$ and $\beta'$ belong to the set $\{ \beta_1, \dotsc, \beta_n \}$ of representatives of isomorphism classes of invertible objects. Thus, by Lemma~\ref{lem:Alg-C-A-1}, it is enough to prove:

\begin{claim}
  $\intch_{\beta} = \intch_{\beta'}$.  
\end{claim}
\begin{proof}
  We first remark the following equation:
  \begin{equation}
    \label{eq:Psi-Phi-claim-1-1}
    G_0^{-1} \circ G(\pi_{\unitobj}(\unitobj))
    = \pi_{\unitobj}(\unitobj) \circ Z(G_0^{-1}) \circ \xi^{G}_{\unitobj}
  \end{equation}
  This follows from the following commutative diagram:
  \begin{equation*}
    \xymatrix{
      G Z(\unitobj)
      \ar@{}[rrrd]|{\text{($\circlearrowleft$ by the definition of $\xi^G$)}}
      \ar[d]_{\xi^G_{\unitobj}}
      \ar[rrr]^(.375){G(\pi_{\unitobj}(\unitobj))}
      & & & G(\unitobj \otimes \unitobj \otimes \unitobj^*)
      \ar[d]^{\text{\eqref{eq:xi-G-def-can-iso} with $F = G$}}
      \ar@{=}[r] & G(\unitobj) \ar[dd]^{G_0^{-1}}\\
      Z G(\unitobj)
      \ar@{}[rrrd]|{\text{($\circlearrowleft$ by the (di)naturality of $\pi$})}
      \ar[d]_{Z(G_0^{-1})}
      \ar[rrr]^(.375){\pi_{G(\unitobj)}(G(\unitobj))}
      & & & G(\unitobj) \otimes G(\unitobj) \otimes G(\unitobj)^*
      \ar[d]^{G_0^{-1} \otimes G_0^{-1} \otimes G_0^*} \\
      Z(\unitobj)
      \ar[rrr]_(.375){\pi_{\unitobj}(\unitobj)}
      & & & \unitobj \otimes \unitobj \otimes \unitobj^*
      \ar@{=}[r] & \unitobj,
    }
  \end{equation*}
  Now the claim is proved as follows:
  \begin{align*}
    \intch_{\beta'}
    & \mathop{=}^{\text{\eqref{eq:def-Psi-beta-2}}}
    \pi_{\unitobj}(\unitobj) \circ Z(G_0^{-1}) \circ \xi^{G}_{\unitobj} \circ h_{R(\unitobj)} \\
    & \mathop{=}^{\text{\eqref{eq:Psi-Phi-claim-1-1}}}
    G_0^{-1} \circ G(\pi_{\unitobj}(\unitobj)) \circ \Phi(\boldsymbol{\beta})_{R(\unitobj)} \\
    & \mathop{=}^{\text{\eqref{eq:lem-chi-2}}}
    G_0^{-1} \circ j_{\unitobj \otimes \unitobj \otimes \unitobj^*} \circ \pi_{V}(\beta \otimes \unitobj)
      = \coev_{\beta}^{-1} \circ \pi_{\unitobj}(\beta \otimes \unitobj) = \intch_{\beta}.
  \end{align*}
  Here, the fourth equality holds since $j \in \Nat_{\otimes}(I^{\beta}, G)$ and $I^{\beta}_0 = \coev_{\beta}$.
\end{proof}

\begin{claim}
  $j = j'$.
\end{claim}
\begin{proof}
  We fix $V, X \in \mathcal{C}$ and consider the following commutative diagram:
  \begin{equation*}
    \xymatrix@C=60pt{
      Z(V^{\beta})
      \ar@{}@<.5ex>[rd]|{\text{($\circlearrowleft$ by the definition of $\zeta^{\beta}$})}
      \ar[d]_{\zeta^{\beta}_V}
      \ar[r]^(.275){\pi_{V^{\beta}}(X^{\beta})}
      & X^{\beta} \otimes V^{\beta} \otimes (X^{\beta})^*
      \ar@{=}[d] \\
      Z(V)
      \ar@{}@<.5ex>[rd]|{\text{($\circlearrowleft$ by the dinaturality of $\pi$})}
      \ar@{=}[d]
      \ar[r]^(.275){\pi_V(\beta \otimes X \otimes \beta^* \otimes \beta)}
      & \beta \otimes X \otimes \beta^* \otimes \beta \otimes V \otimes \beta^*
      \otimes \beta^{**} \otimes X^{*} \otimes \beta^*
      \ar[d]^{\id \otimes \id \otimes \eval_{\beta} \otimes \id \otimes \coev_{\beta^*} \otimes \id \otimes \id}\\
      Z(V)
      \ar@{}@<.5ex>[rd]|{\text{($\circlearrowleft$ by Lemma~\ref{lem:chi-2})}}
      \ar[d]_{h_{R(V)}}
      \ar[r]^(.275){\pi_V(\beta \otimes X)}
      & (X \otimes V \otimes X^{*})^{\beta}
      \ar[d]^{j_{X \otimes V \otimes X^*}} \\
      F Z(V)
      \ar@{}@<.5ex>[rd]|{\text{($\circlearrowleft$ by the definition of $\xi^F$})}
      \ar[d]_{\xi^F_V}
      \ar[r]^(.275){G(\pi_V(X))}
      & G(X \otimes V \otimes X^*)
      \ar[d]^{\text{\eqref{eq:xi-G-def-can-iso} with $F = G$}} \\
      Z G(V)
      \ar[r]^(.275){\pi_{G(V)}(G(X))}
      & G(X) \otimes G(V) \otimes G(X)^*. \\
    }
  \end{equation*}
  By~\eqref{eq:def-Psi-j}, the composition along the first column is $Z(j'_V)$. By Lemma~\ref{lem:mon-nat-inv}, the composition along the second column is $j_X \otimes j_V \otimes (j_X^{-1})^*$. Thus we have
  \begin{align*}
    \pi_{G(V)}(G(X)) \circ Z(j'_V)
    & = (j_X \otimes j_V \otimes (j_X^{-1})^*) \circ \pi_{V^{\beta}}(X^{\beta}) \\
    & = (\id_{G(X)} \otimes j_V \otimes \id_{G(X)}) \circ \pi_{V^{\beta}}(G(X)) \\
    & = \pi_{G(V)}(G(X)) \circ Z(j_V).
  \end{align*}
  Since $G: \mathcal{C} \to \mathcal{C}$ is an equivalence, we have $Z(j_V) = Z(j'_V)$ by the universal property of $\pi_{G(V)}$. Since $Z$ is faithful, we have $j_V = j'_V$.
\end{proof}

Therefore $\Psi_G \Phi_G$ is the identity map.

\subsection{Computation of $\Phi_G \Psi_G$}
\label{subsec:compute-Phi-Psi}

Finally, we show that the composition $\Phi_G \Psi_G$ is the identity map. We fix $h \in N_G$ and set
\begin{equation*}
  \boldsymbol{\beta} = (\beta, j) = \Psi_G(h)
  \quad \text{and} \quad
  h' = \Phi_G(\boldsymbol{\beta}).
\end{equation*}

\begin{claim}
  $h_{R(V)} = h'_{R(V)}$ for all objects $V \in \mathcal{C}$.
\end{claim}
\begin{proof}
  We fix $V \in \mathcal{C}$. Since $G$ is an equivalence, we have
  \begin{equation*}
    G Z(V) = \int_{X \in \mathcal{C}} G(X \otimes V \otimes X^*)
  \end{equation*}
  with the universal dinatural transformation $G(\pi_{V}(-))$. Thus, to prove this claim, it is sufficient to show that the equation
  \begin{equation*}
    G(\pi_V(X)) \circ h_{R(V)} = G(\pi_V(X)) \circ h'_{R(V)}
  \end{equation*}
  holds for all objects $X \in \mathcal{C}$. For simplicity, we set
  \begin{equation*}
    p_X := (j_X \otimes \id_{\beta}) \circ (\id_{\beta} \otimes \id_X \otimes \eval_{\beta}^{-1})
  \end{equation*}
  for $X \in \mathcal{C}$. Then we have the the following commutative diagram:
  \begin{equation*}
    \xymatrix{
      Z(V)
      \ar@{}@<.5ex>[rrrd]|{\text{($\circlearrowleft$ by the dinaturality of $\pi_V(X)$ in $X$})}
      \ar@{=}[d]
      \ar[rrr]^(.3){\pi_{V}(\beta \otimes X)}
      & & & \beta \otimes X \otimes V \otimes X^* \otimes \beta^*
      \ar[d]^{p_X \otimes \id_V \otimes (p^{-1})^*} \\
      Z(V)
      \ar@{}@<.5ex>[rrrd]|{\text{($\circlearrowleft$ by the definition of $\zeta^{\beta}$})}
      \ar[d]_{(\zeta_V^{\beta})^{-1}}
      \ar[rrr]^(.3){\pi_{V}(G(X) \otimes \beta)}
      & & & G(X) \otimes \beta \otimes V \otimes \beta^* \otimes G(X)^*
      \ar@{=}[d] \\
      Z(V^{\beta})
      \ar@{}@<.5ex>[rrrd]|{\text{($\circlearrowleft$ by the naturality of $\pi_V(X)$ in $V$})}
      \ar[d]_{Z(j_V)}
      \ar[rrr]^(.3){\pi_{V^{\beta}}(G(X))}
      & & & G(X) \otimes \beta \otimes V \otimes \beta^* \otimes G(X)^*
      \ar[d]^{\id \otimes j_V \otimes \id} \\
      Z G(V)
      \ar@{}@<.5ex>[rrrd]|{\text{($\circlearrowleft$ by the definition of $\xi^F$})}
      \ar[d]_{(\xi^F_V)^{-1}}
      \ar[rrr]^(.3){\pi_{G(V)}(G(X))}
      & & & G(X) \otimes G(V) \otimes G(X)^*
      \ar[d]^{\text{\eqref{eq:xi-G-def-can-iso} with $F = G$}} \\
      F Z(V)
      \ar[rrr]^(.3){G(\pi_V(X))}
      & & & G(X \otimes V \otimes X^*)
    }
  \end{equation*}
  By Lemma~\ref{lem:mon-nat-inv}, the composition along the right column is $j_{X \otimes V \otimes X^*}$. Thus, by the above commutative diagram, we have
  \begin{gather*}
    G(\pi_{V}(X)) \circ h_{R(V)}
    \mathop{=}^{\text{\eqref{eq:def-Psi-j}}} G(\pi_{V}(X)) \circ (\xi_V^F)^{-1} \circ h_{R(V)} \circ (\zeta_V^{\beta})^{-1} \\
    = j_{X \otimes V \otimes X^*} \circ \pi_{V}(\beta \otimes X)
    \mathop{=}^{\text{\eqref{eq:lem-chi-2}}}
    G(\pi_{V}(X)) \circ h'_{R(V)}
  \end{gather*}
  for all objects $X \in \mathcal{C}$. The proof is done.
\end{proof}

Since every component of the unit of the adjunction $U \dashv R$ is a monomorphism, every object $\mathbf{X} \in \mathcal{Z}(\mathcal{C})$ fits into an exact sequence of the form $0 \to \mathbf{X} \to R(V) \to R(W)$ in $\mathcal{Z}(\mathcal{C})$. Thus, by the above claim and the standard argument, we obtain $h_{\mathbf{X}} = h'_{\mathbf{X}}$ for all objects $\mathbf{X} \in \mathcal{Z}(\mathcal{C})$. Therefore $\Phi_G \Psi_G$ is the identity map. We have completed the proof of Theorem~\ref{thm:main-thm}.

\section{Further applications and remarks}
\label{sec:applications}


\subsection{Functorial property of $\Phi$}

Let $\mathcal{C}$ be a finite tensor category. For tensor autoequivalences $F$ and $G$ of $\mathcal{C}$, we define $\mathscr{J}(F, G)$ and $J(F, G)$ in the same way as Section~\ref{sec:main-result}. The main result of Section~\ref{sec:main-result} states that there is a well-defined bijection
\begin{equation*}
  \Phi_{F, G}: \mathscr{J}(F, G) \to \Nat_{\otimes}(\widetilde{F}, \widetilde{G}), \quad [\boldsymbol{\beta}] \mapsto \Phi(\boldsymbol{\beta}).
\end{equation*}
Let $H$ be another tensor autoequivalence of $\mathcal{C}$. For $\boldsymbol{\alpha} = (\alpha, i) \in \mathscr{J}(G, H)$ and $\boldsymbol{\beta} = (\beta, j) \in \mathscr{J}(F, G)$, we define their tensor product by $\boldsymbol{\alpha} \otimes \boldsymbol{\beta} = (\alpha \otimes \beta, h)$, where
\begin{equation}
  \label{eq:functorial-Phi-1}
  h = \left(
    I^{\alpha \otimes \beta} F
    \xrightarrow{\quad \id \quad}
    I^{\alpha} I^{\beta} F
    \xrightarrow{\quad I^{\alpha}(j) \quad}
    I^{\alpha} G
    \xrightarrow{\quad i \quad}
    H
  \right).
\end{equation}
It is easy to see that the following operation is well-defined:
\begin{equation}
  \label{eq:functorial-Phi-2}
  J(G, H) \times J(F, G) \to J(F, H),
  \quad ([\boldsymbol{\alpha}], [\boldsymbol{\beta}]) \mapsto [\boldsymbol{\alpha} \otimes \boldsymbol{\beta}]
\end{equation}

\begin{lemma}
  \label{lem:Phi-functorial}
  For $\boldsymbol{\alpha} = (\alpha, i) \in \mathscr{J}(G, H)$ and $\boldsymbol{\beta} = (\beta, j) \in \mathscr{J}(F, G)$, we have
  \begin{equation*}
    \Phi(\boldsymbol{\alpha} \otimes \boldsymbol{\beta}) = \Phi(\boldsymbol{\alpha}) \circ \Phi(\boldsymbol{\beta}).
  \end{equation*}
\end{lemma}

Namely, the following diagram is commutative:
\begin{equation*}
  \xymatrix@C=64pt{
    J(G, H) \times J(F, G)
    \ar[r]^(.6){\text{\eqref{eq:functorial-Phi-2}}}
    \ar[d]_{\Phi_{G,H} \times \Phi_{F,G}}
    & J(F, H) \ar[d]^{\Phi_{F,H}} \\
    \Nat_{\otimes}(\widetilde{G}, \widetilde{H})
    \times \Nat_{\otimes}(\widetilde{F}, \widetilde{G})
    \ar[r]^(.6){\text{composition}}
    & \Nat_{\otimes}(\widetilde{F}, \widetilde{H}).
  }
\end{equation*}

\begin{proof}
  For all $\mathbf{V} = (V, \sigma_V) \in \mathcal{Z}(\mathcal{C})$, we have
  \begin{equation*}
    \Phi(\boldsymbol{\alpha})_{\mathbf{V}}
    \circ \Phi(\boldsymbol{\beta})_{\mathbf{V}}
    = \!\!\!\!\! {\knotholesize{6pt} \xy /r1pc/:
      (0,3.5)="P1", (0,-3.5)="P2",
      "P1" *+!D{\scriptstyle {F(V)}},
      "P1"; p-(0, .5) **\dir{-}
      ?> *+!U{\makebox[2pc]{$\scriptstyle {\Phi(\boldsymbol{\beta})_{\mathbf{V}}}$}} *\frm{-}="BX1",
      "BX1"!D; p-(0, 1.25) **\dir{-} ?> \pushrect{1}{1.5},
      \vtwist~{s0}{s1}{s2}{s3},
      s1; p+(.75,0) \mycap{.75}
      ?>; \makecoord{p}{s2} **\dir{-},
      s3 *+!U{\makebox[2pc]{$\scriptstyle {i_V}$}} *\frm{-}="BX2",
      "BX2"!D; \makecoord{p}{"P2"} **\dir{-}
      ?> *+!U{\scriptstyle {H(V)}},
      s0+(-.75, .75)="T3",
      "BX2"!D!R+(.5, -.5)="T4",
      "T3"; "T4" **\frm{.}, "T4" *+{},
      \makecoord{"T3"}{"T4"} *+!U{\scriptstyle {\Phi(\boldsymbol{\alpha})_{\mathbf{V}}}}
      \endxy}
    = \!\!\! {\knotholesize{8pt} \xy /r1pc/:
      (0,3.5)="P1", (0,-3.5)="P2",
      "P1" *+!D{\scriptstyle {F(V)}},
      "P1"; p-(0, .75) **\dir{-} ?> \pushrect{2}{2}
      \vtwist~{s0}{s1}{s2}{s3},
      s3 *+!U{\makebox[2pc]{$\scriptstyle \Phi(\boldsymbol{\beta})_{\mathbf{V}}$}} *\frm{-}="BX1",
      "BX1"!D; p-(0,1.25) **\dir{-}
      ?> *+!U{\makebox[2pc]{$\scriptstyle {i_V}$}} *\frm{-}="BX2",
      "BX2"!U!L(.5)="T1",
      s2; \makecoord{p}{"BX1"!D} **\dir{-}
      ?>; "T1" \mybend{.6},
      "BX2"!U!R(.5)="T2",
      s1; p+(2,0) \mycap{.75}
      ?>; \makecoord{p}{"BX1"!D} **\dir{-}
      ?>; "T2" \mybend{.6},
      "BX2"!D; \makecoord{p}{"P2"} **\dir{-}
      ?> *+!U{\scriptstyle {H(V)}}
      \endxy}
    \,\, = \!\!\!\!\! {\knotholesize{6pt} \xy /r1pc/:
      (0,3.5)="P1", (0,-3.5)="P2",
      "P1" *+!D{\scriptstyle {F(V)}},
      "P1"; p-(0,1) **\dir{-} ?> \pushrect{1}{1},
      \vtwist~{s0}{s1}{s2}{s3},
      s3 \pushrect{.75}{1},
      \vtwist~{s0}{s1}{s2}{s3},
      s1; p+(.75,0) \mycap{.75}
      ?>; \makecoord{p}{s2} **\dir{-},
      s3 *+!U{\makebox[2pc]{$\scriptstyle {j_V}$}} *\frm{-}="BX1",
      "BX1"!D; p-(0,1.25) **\dir{-}
      ?> *+!U{\makebox[2pc]{$\scriptstyle {i_V}$}} *\frm{-}="BX2",
      "BX2"!U!L(.5)="T1",
      s6; \makecoord{p}{"BX1"!D} **\dir{-}
      ?>; "T1" \mybend{.6},
      "BX2"!U!R(.5)="T2",
      s5; p+(2.5,0) \mycap{.75}
      ?>; \makecoord{p}{"BX1"!D} **\dir{-}
      ?>; "T2" \mybend{.6},
      "BX2"!D; \makecoord{p}{"P2"} **\dir{-}
      ?> *+!U{\scriptstyle {H(V)}}
      \endxy}
    = \Phi(\boldsymbol{\alpha} \otimes \boldsymbol{\beta})_{\mathbf{V}}.
  \end{equation*}
  Here, the second equality follows from that $\Phi(\boldsymbol{\beta})_{\mathbf{V}}$ is a morphism in $\mathcal{Z}(\mathcal{C})$, and the others follow from the definition of $\Phi$. The proof is done.
\end{proof}

For a finite tensor category $\mathcal{B}$, we set $\Aut_{\otimes}(\id_{\mathcal{B}}) := \Nat_{\otimes}(\id_{\mathcal{B}}, \id_{\mathcal{B}})$. Since $\mathcal{B}$ is rigid, the set $\Aut_{\otimes}(\id_{\mathcal{B}})$ is in fact a group with respect to the composition (see
Subsection \ref{lem:mon-nat-inv}). We also denote by $\Inv(\mathcal{B})$ the group of isomorphism classes of invertible objects of $\mathcal{B}$. By using our main result, we obtain:

\begin{theorem}
  \label{thm:Aut-id-ZC}
  $\Inv(\mathcal{Z}(\mathcal{C})) \cong \Aut_{\otimes}(\id_{\mathcal{Z}(\mathcal{C})})$ as groups.
\end{theorem}
\begin{proof}
  Let $\boldsymbol{\beta} = (\beta, \sigma_{\beta}) \in \mathcal{Z}(\mathcal{C})$ be an invertible object. We note that $\beta \in \mathcal{C}$ is an invertible object. We define $j \in \Nat(I^{\beta}, \id_{\mathcal{C}})$ by
  \begin{equation*}
    j_X = \left(
      \beta \otimes X \otimes \beta^*
      \xrightarrow{\quad \sigma_{\beta}(X) \quad}
      X \otimes \beta \otimes \beta^*
      \xrightarrow{\quad \id_X \otimes \coev_{\beta}^{-1} \quad}
      X \right)
  \end{equation*}
  for $X \in \mathcal{C}$. Then $j$ is a monoidal natural transformation. To see this, we note
  \begin{equation}
    \label{eq:inv-ZC-aut-id-pf-1}
    \coev_{\beta}^{-1} \otimes \id_{\beta}
    = (\id_{\beta} \otimes \eval_{\beta})(\coev_{\beta} \otimes \id_{\beta})
    (\coev_{\beta}^{-1} \otimes \id_{\beta})
    = \id_{\beta} \otimes \eval_{\beta}.
  \end{equation}
  Set $c = \coev_{\beta}^{-1}$ for simplicity. By graphical calculus, we compute
  \begin{align*}
    j_X \otimes j_Y
    & = {\knotholesize{8pt} \xy /r1pc/:
      (0,2)="P1"; p-(0,4)="P2",
      "P1" \pushrect{1}{1}, \vtwist~{s0}{s1}{s2}{s3},
      s0 *+!D{\scriptstyle \beta},
      s1 *+!D{\scriptstyle X_{}},
      s2; "P2" **\dir{-} ?> *+!U{\scriptstyle X_{}},
      s3; p+(1,0) **\dir{} ?(.5)="T1" ?>="T2",
      "T1" *+!U{\makebox[1.5pc]{$\scriptstyle c$}} *\frm{-},
      "T2"; \makecoord{p}{"P1"} **\dir{-}
      ?> *+!D{\scriptstyle \,\,\beta\smash{{}^*}},
      (3.25,2)="P1"; p-(0,4)="P2",
      "P1" \pushrect{1}{1}, \vtwist~{s0}{s1}{s2}{s3},
      s0 *+!D{\scriptstyle \beta},
      s1 *+!D{\scriptstyle Y_{}},
      s2; "P2" **\dir{-} ?> *+!U{\scriptstyle Y_{}},
      s3; p+(1,0) **\dir{} ?(.5)="T1" ?>="T2",
      "T1" *+!U{\makebox[1.5pc]{$\scriptstyle c$}} *\frm{-},
      "T2"; \makecoord{p}{"P1"} **\dir{-}
      ?> *+!D{\scriptstyle \,\,\beta\smash{{}^*}}
      \endxy}
      = {\knotholesize{8pt} \xy /r1pc/:
      (0,2)="P1"; p-(0,4)="P2",
      "P1" \pushrect{1}{1}, \vtwist~{s0}{s1}{s2}{s3},
      s0 *+!D{\scriptstyle \beta},
      s1 *+!D{\scriptstyle X_{}},
      s2; "P2" **\dir{-} ?> *+!U{\scriptstyle X_{}},
      s3; p+(1,0) **\dir{} ?(.5)="T1" ?>="T2",
      "T1" *+!U{\makebox[1.5pc]{$\scriptstyle c$}} *\frm{-}="BX1",
      "T2"; \makecoord{p}{"P1"} **\dir{-}
      ?> *+!D{\scriptstyle \,\,\beta\smash{{}^*}},
      (3.25,2)="P1"; p-(0,4)="P2",
      "P1"; p-(0,2) **\dir{-}
      ?> \pushrect{1}{1}, \vtwist~{s0}{s1}{s2}{s3},
      "P1" *+!D{\scriptstyle \beta},
      s1; \makecoord{p}{"P1"} **\dir{-}
      ?> *+!D{\scriptstyle Y_{}},
      s2; "P2" **\dir{-} ?> *+!U{\scriptstyle Y_{}},
      s3; p+(1,0) **\dir{} ?(.5)="T1" ?>="T2",
      "T1" *+!U{\makebox[1.5pc]{$\scriptstyle c$}} *\frm{-},
      "T2"; \makecoord{p}{"P1"} **\dir{-}
      ?> *+!D{\scriptstyle \,\,\beta\smash{{}^*}},
      "BX1"!U!L+(-.25, .25)="T1", s0+(.25, 0)="T2",
      "T1"; "T2" **\frm{.}
      \endxy}
      \mathop{=}^{\text{\eqref{eq:inv-ZC-aut-id-pf-1}}}
      {\knotholesize{8pt} \xy /r1pc/:
      (0,2)="P1"; p-(0,4)="P2",
      "P1" \pushrect{1}{1}, \vtwist~{s0}{s1}{s2}{s3},
      s0 *+!D{\scriptstyle \beta},
      s1 *+!D{\scriptstyle X_{}},
      s2; "P2" **\dir{-} ?> *+!U{\scriptstyle X_{}},
      "P1"+(2,0); p+(1,0) \mycap{-.75}
      ?< *+!D{\scriptstyle \,\,\beta\smash{{}^*}},
      ?> *+!D{\scriptstyle \beta}
      ?>+(1,0)="P3",
      "P3" *+!D{\scriptstyle Y_{}},
      "P3"; p-(0,1) **\dir{-}
      ?>-(3,0) \pushrect{3}{2}, \vtwist~{s0}{s1}{s2}{s3},
      s3; p+(1,0) **\dir{} ?(.5)="T1" ?>="T2",
      "T1" *+!U{\makebox[1.5pc]{$\scriptstyle c$}} *\frm{-},
      "T2"; \makecoord{p}{"P1"} **\dir{-}
      ?> *+!D{\scriptstyle \,\,\beta\smash{{}^*}},
      s2; "P2"+(1,0) \mybend{.75} ?> *+!U{\scriptstyle Y}
      \endxy} \\
    & = j_{X \otimes Y} \circ I_2^{\beta}(X, Y)
  \end{align*}
  for all $X, Y \in \mathcal{C}$. The equation $j_{\unitobj} \circ I_{0}^{\beta} = \id_{\unitobj}$ is trivial. Thus $j \in \Nat_{\otimes}(I^{\beta}, \id_{\mathcal{C}})$.

  Let $\boldsymbol{\alpha}$ and $\boldsymbol{\beta}$ be invertible objects of $\mathcal{Z}(\mathcal{C})$, and let $i: I^{\alpha} \to \id_{\mathcal{C}}$ and $j: I^{\beta} \to \id_{\mathcal{C}}$ be the monoidal natural transformations associated to $\boldsymbol{\alpha}$ and $\boldsymbol{\beta}$, respectively. Then the monoidal natural transformation $h$ associated to $\boldsymbol{\alpha} \otimes \boldsymbol{\beta}$ is given by
  \begin{align*}
    h_X
    & = {\knotholesize{8pt} \xy /r1pc/:
      (0,2)="P1"; p-(0,4)="P2",
      "P1" \pushrect{2}{2}, \vtwist~{s0}{s1}{s2}{s3},
      s0 *+!D{\scriptstyle \alpha \otimes \beta},
      s1 *+!D{\scriptstyle X_{}},
      s2; "P2" **\dir{-} ?> *+!U{\scriptstyle X_{}},
      s3; p+(2,0) **\dir{} ?(.5)="T1" ?>="T2",
      "T1" *+!U{\makebox[2.5pc]{$\scriptstyle \coev_{\alpha \otimes \beta}^{-1}$}} *\frm{-},
      "T2"; \makecoord{p}{"P1"} **\dir{-}
      ?> *+!D{\scriptstyle \beta\smash{{}^*} \otimes \alpha\smash{{}^*}}
      \endxy}
      = {\knotholesize{8pt} \xy /r1pc/:
      (0,2)="P1"; p-(0,4)="P2",
      "P1" \pushrect{1.25}{1.25}, \vtwist~{s0}{s1}{s2}{s3},
      s0 *+!D{\scriptstyle \beta},
      s1 *+!D{\scriptstyle X_{}},
      s3; p+(1.25,0) **\dir{} ?(.5)="T1" ?>="T2",
      "T1" *+!U{\makebox[2pc]{$\scriptstyle \coev_{\beta}^{-1}$}} *\frm{-},
      "T2"; \makecoord{p}{"P1"} **\dir{-}
      ?>="T3" *+!D{\scriptstyle \,\, \beta\smash{{}^*}},
      s2; p-(0,.5) **\dir{-} ?> -(1.25, 0)
      \pushrect{1.25}{1.25}, \vtwist~{s0}{s1}{s2}{s3},
      s0; \makecoord{p}{"P1"} **\dir{-} ?> *+!D{\scriptstyle \alpha_{}},
      s2; \makecoord{p}{"P2"} **\dir{-} ?> *+!U{\scriptstyle X},
      "T3"+(1,0); \makecoord{p}{s3} **\dir{-}
      ?< *+!D{\scriptstyle \,\, \alpha_{}\smash{{}^*}}
      ?>; s3 **\dir{} ?(.5)
      *+!U{\makebox[4pc]{$\scriptstyle \coev_{\alpha}^{-1}$}} *\frm{-},
      \endxy}
      = {\xy /r1pc/:
      (0,2)="P1"; p-(0,4)="P2",
      "P1"+(0,0)="T1" *+!D{\scriptstyle \alpha_{}},
      "P1"+(1,0)="T2" *+!D{\scriptstyle \beta},
      "P1"+(2,0)="T3" *+!D{\scriptstyle X_{}},
      "P1"+(3,0)="T4" *+!D{\scriptstyle \beta\smash{^*}},
      "P1"+(4,0)="T5" *+!D{\scriptstyle \alpha_{}\smash{^*}},
      "T3"; p-(0,1) **\dir{-}
      ?> *+!U{\scriptstyle j_X} *\frm{-}="BX1",
      "T2"; "BX1"!U!L(.6) \mybend{.5},
      "T4"; "BX1"!U!R(.6) \mybend{.5},
      "BX1"!D; p-(0,.5) **\dir{-}
      ?> *+!U{\scriptstyle \quad i_X \quad} *\frm{-}="BX2",
      "T1"; "BX2"!U!L(.6) \mybend{.5},
      "T5"; "BX2"!U!R(.6) \mybend{.5},
      "BX2"!D; \makecoord{p}{"P2"} **\dir{-}
      ?> *+!U{\scriptstyle X}
      \endxy} \\
    & = \text{(the right-hand side of \eqref{eq:functorial-Phi-1})}.
  \end{align*}
  Namely, we obtain a map
  \begin{equation}
    \label{eq:inv-ZC-and-J}
    (\text{the class of invertible objects of $\mathcal{Z}(\mathcal{C})$}) \to \mathscr{J}(\id_{\mathcal{C}}, \id_{\mathcal{C}})
  \end{equation}
  preserving the tensor product. Conversely, given an element $(\beta, j) \in \mathscr{J}(\id_{\mathcal{C}}, \id_{\mathcal{C}})$, we define a natural transformation $\sigma_{\beta}$ by
  \begin{equation*}
    \sigma_{\beta}(X) = \left(
      \beta \otimes X
      \xrightarrow{\quad \id_{\beta} \otimes \id_X \otimes \eval_{\beta}^{-1} \quad}
      \beta \otimes X \otimes \beta^* \otimes \beta
      \xrightarrow{\quad j_X \quad} X \otimes \beta
    \right)
  \end{equation*}
  for $X \in \mathcal{C}$. It is routine to check that $(\beta, \sigma_{\beta})$ is an invertible object of $\mathcal{Z}(\mathcal{C})$ and this construction gives an inverse of the map~\eqref{eq:inv-ZC-and-J}.

  The map~\eqref{eq:inv-ZC-and-J} induces a bijection $\Inv(\mathcal{Z}(\mathcal{C})) \cong J(\id_{\mathcal{C}}, \id_{\mathcal{C}})$ preserving the tensor product. By the above argument and Lemma~\ref{lem:Phi-functorial}, the composition
  \begin{equation*}
    \Inv(\mathcal{Z}(\mathcal{C}))
    \xrightarrow{\quad \cong \quad} J(\id_{\mathcal{C}}, \id_{\mathcal{C}})
    \xrightarrow{\quad \Phi \quad} \Aut_{\otimes}(\id_{\mathcal{C}})
  \end{equation*}
  is an isomorphism of groups. The proof is done.
\end{proof}

\begin{remark}
  The isomorphism $\Inv(\mathcal{Z}(\mathcal{C}) \cong \Aut_{\otimes}(\id_{\mathcal{Z}(\mathcal{C})})$ of Theorem~\ref{thm:Aut-id-ZC} is expressed as follows: Let $\boldsymbol{\beta} = (\beta, \sigma_{\beta})$ be an invertible object of $\mathcal{Z}(\mathcal{C})$, and let $\Omega_{\boldsymbol{\beta}}$ be the element of $\Aut_{\otimes}(\id_{\mathcal{Z}(\mathcal{C})})$ corresponding to $[\boldsymbol{\beta}]$. If we denote the braiding of $\mathcal{Z}(\mathcal{C})$ by $\mathbf{c}$, then we have
  \begin{equation*}
    \Omega_{\boldsymbol{\beta}}(\mathbf{V}) \otimes \id_{\beta}
    =
    {\knotholesize{8pt} \xy /r1pc/:
      (0,2)="P1"; p-(0,4)="P2",
      "P1" *+!D{\scriptstyle V_{}}; p-(0,.25) **\dir{-}
      ?> \pushrect{1.5}{1}, \vtwist~{s0}{s1}{s2}{s3},
      s2 \pushrect{1.5}{1}, \vtwist~{s0}{s1}{s2}{s3},
      s5; p+(1,0) \mycap{.75}
      ?>; \makecoord{p}{s3} **\dir{-}
      ?>; s3 **\dir{} ?(.5)
      *+!U{\scriptstyle \coev_{\beta}^{-1}} *\frm{-},
      s2; \makecoord{p}{"P2"} **\dir{-}
      ?> *+!U{\scriptstyle V},
      s5+(2.5,0); \makecoord{p}{"P1"}; \makecoord{p}{"P2"} **\dir{-}
      ?< *+!D{\scriptstyle \beta}
      ?> *+!U{\scriptstyle \beta}
      \endxy}
    \mathop{=}^{\text{\eqref{eq:inv-ZC-aut-id-pf-1}}}
    {\knotholesize{8pt} \xy /r1pc/:
      (0,2)="P1"; p-(0,4)="P2",
      "P1" *+!D{\scriptstyle V_{}}; p-(0,.25) **\dir{-}
      ?> \pushrect{1.5}{1.5}, \vtwist~{s0}{s1}{s2}{s3},
      s2 \pushrect{1.5}{1.5}, \vtwist~{s0}{s1}{s2}{s3},
      s5; \makecoord{p}{"P1"} **\dir{-} ?> *+!D{\scriptstyle \beta},
      s2; \makecoord{p}{"P2"} **\dir{-} ?> *+!U{\scriptstyle V_{}},
      s3; \makecoord{p}{"P2"} **\dir{-} ?> *+!U{\scriptstyle \beta},
      \endxy}
    = \mathbf{c}_{\boldsymbol{\beta}, \mathbf{V}} \circ \mathbf{c}_{\mathbf{V}, \boldsymbol{\beta}}
  \end{equation*}
  for all $\mathbf{V} = (V, \sigma_V) \in \mathcal{Z}(\mathcal{C})$. Namely, $\Omega_{\boldsymbol{\beta}}$ is the monodromy \cite[Lemma 2.4]{MR1759389} around the invertible object $\boldsymbol{\beta}$.
\end{remark}

\subsection{Rosenberg-Zelinsky exact sequence}

Given a finite tensor category $\mathcal{B}$, we denote by $\Aut_{\otimes}(\mathcal{B})$ the group of isomorphism classes of tensor autoequivalences of $\mathcal{B}$. If, in addition, $\mathcal{B}$ is braided, then we denote by $\Aut_{\otimes}^{\mathrm{br}}(\mathcal{B})$ the group of braided tensor autoequivalences of $\mathcal{B}$.

Let $\mathcal{C}$ be a finite tensor category. The Brauer-Picard group $\mathrm{BrPic}(\mathcal{C})$ of $\mathcal{C}$ is the group of equivalence classes of invertible $\mathcal{C}$-bimodule categories \cite{MR2677836}. There is a canonical isomorphism $\mathrm{BrPic}(\mathcal{C}) \cong \Aut_{\otimes}^{\mathrm{br}}(\mathcal{Z}(\mathcal{C}))$ of groups \cite{MR3107567}, and thus we identify them. By Theorem~\ref{thm:main-thm}, we obtain the following theorem, which has been known in the semisimple case ({\it cf}. a categorification of the Rosenberg-Zelinsky exact sequence for fusion categories given in \cite{MR3606516}).

\begin{theorem}
  \label{thm:rose-zeli}
  For a finite tensor category $\mathcal{C}$, there is an exact sequence
  \begin{equation*}
    \newcommand{\XARR}[1]{\xrightarrow{\mathrm{#1}}}
    1 \to \Aut_{\otimes}(\id_{\mathcal{C}})
    \XARR{(i)} \Inv(\mathcal{Z}(\mathcal{C}))
    \XARR{(ii)} \Inv(\mathcal{C})
    \XARR{(iii)} \Aut_{\otimes}(\mathcal{C})
    \XARR{(iv)} \mathrm{BrPic}(\mathcal{C}).
  \end{equation*}
  Here, the group homomorphism \textup{(i)} sends $\xi \in G(\mathcal{C})$ to the isomorphism class of the object $(\unitobj, \sigma_{\xi}) \in \mathcal{Z}(\mathcal{C})$, where the half-braiding $\sigma_{\xi}$ is defined by
  \begin{equation*}
    \sigma_{\xi}(X) := \left(
      \unitobj \otimes X \xrightarrow{\quad \id \quad} X
      \xrightarrow{\quad \xi_X \quad}
      X \xrightarrow{\quad \id \quad} X \otimes \unitobj
    \right)
  \end{equation*}
  for $X \in \mathcal{C}$. The group homomorphism \textup{(ii)} is given by
  \begin{equation*}
    \Inv(\mathcal{Z}(\mathcal{C})) \to \Inv(\mathcal{C}),
    \quad [(\beta, \sigma_{\beta})] \mapsto [\beta].
  \end{equation*}
  The group homomorphism \textup{(iii)} is given by
  \begin{equation*}
    \Inv(\mathcal{C}) \to \Aut_{\otimes}(\mathcal{C}),
    \quad [\beta] \mapsto [I^{\beta}].
  \end{equation*}
  The group homomorphism \textup{(iv)} is given by
  \begin{equation*}
    \Aut_{\otimes}(\mathcal{C})
    \to \mathrm{BrPic}(\mathcal{C}) = \Aut_{\otimes}^{\mathrm{br}}(\mathcal{Z}(\mathcal{C})),
    \quad [F] \mapsto [\widetilde{F}].
  \end{equation*}
\end{theorem}
\begin{proof}
  It is obvious that the group homomorphism (i) is injective. Since the image of (i) consists of the isomorphism classes of objects of $\mathcal{Z}(\mathcal{C})$ of the form $(\unitobj, \sigma)$ for some $\sigma$, the sequence is exact at $\Inv(\mathcal{Z}(\mathcal{C}))$. For an invertible object $\beta \in \mathcal{C}$, the tensor functor $I^{\beta}$ is isomorphic to $\id_{\mathcal{C}}$ if and only if $(\beta, \sigma) \in \mathcal{Z}(\mathcal{C})$ for some $\sigma$ (see the proof of Theorem~\ref{thm:Aut-id-ZC}). This implies the exactness at $\Inv(\mathcal{C})$. By Theorem~\ref{thm:main-thm}, $[F] \in \Aut_{\otimes}(\mathcal{C})$ belongs to the kernel of (iv) if and only if $F \cong I^{\beta}$ for some invertible object $\beta \in \mathcal{C}$, that is, $F$ belongs to the image of (iii). Thus the sequence is exact at $\Aut_{\otimes}(\mathcal{C})$. The proof is done.
\end{proof}



\def\cprime{$'$}

\end{document}